\documentclass{amsart}
\usepackage{amsmath}
\usepackage{amsthm}
\usepackage{amssymb}
\usepackage[usenames,dvipsnames]{xcolor}
\usepackage{tikz-cd}
\usepackage{etoolbox}
\usepackage{enumitem}

\usepackage{tikz}
\usetikzlibrary{calc}

\newtheorem{thm}{Theorem}[section]
\newtheorem{theorem}[thm]{Theorem}
\newtheorem{corollary}[thm]{Corollary}

\newtheorem{lemma}[thm]{Lemma}
\newtheorem{prop}[thm]{Proposition}
\newtheorem{proposition}[thm]{Proposition}

\theoremstyle{definition}
\newtheorem{defn}[thm]{Definition}

\newtheorem{example}[thm]{Example}
\newtheorem{remark}[thm]{Remark}

\makeatletter
\newcommand*\leftdash{\rotatebox[origin=c]{-45}{$\dabar@\dabar@\dabar@$}}
\newcommand*\rightdash{\rotatebox[origin=c]{45}{$\dabar@\dabar@\dabar@$}}
\makeatother

\newcommand{\bks}{\mathcal{S}^{I \times \mathbf{T}(\mathcal{O})}}
\newcommand{\X}{\mathbf{X}(k)}
\newcommand{\kl}{\mathcal{A}_{\mathcal{P}}}
\newcommand{\waff}{\tilde{W}}

\newcommand{\uphm}[1]{M_{d,{#1}}}

\newcommand{\aH}{\tilde{\mathcal{H}}}

\newcommand{\fH}{\mathcal{H}}
\newcommand{\aW}{\tilde{W}}
\newcommand{\Ql}{\overline{\mathbb{Q}}_{\ell}}
\newcommand{\Qlu}{\underline{\overline{\mathbb{Q}}}_{\ell}}

\newcommand{\flinf}{\mathrm{Perv}(\mathcal{F}\ell^{\frac{\infty}{2}})^{I^0}}
\AtBeginEnvironment{theorem}{\setlist{before=\leavevmode, nosep}}

\definecolor{color1}{HTML}{fb8072}
\definecolor{color2}{HTML}{fdb462}
\definecolor{color3}{HTML}{ffed6f}
\definecolor{color4}{HTML}{b3de69}
\definecolor{color5}{HTML}{80b1d3}
\definecolor{color6}{HTML}{bebada}
\definecolor{color7}{HTML}{8dd3c7}
\definecolor{color8}{HTML}{fccde5}
\definecolor{color9}{HTML}{d9d9d9}
\definecolor{color10}{HTML}{bc80bd}
\definecolor{color11}{HTML}{ffffb3}
\definecolor{color12}{HTML}{ccebc5}

\usepackage[utf8]{inputenc}
\numberwithin{equation}{section}

\title[Kazhdan-Laumon Category $\mathcal{O}$, Schwartz space, and semi-infinite flags]{Kazhdan-Laumon Category $\mathcal{O}$, Braverman-Kazhdan Schwartz space, and the semi-infinite flag variety}
\author{Calder Morton-Ferguson}
\date{February 10, 2025}

\begin{document}

\begin{abstract}
We define and study an analogue of Category $\mathcal{O}$ in the context of Kazhdan and Laumon's gluing construction for perverse sheaves on the basic affine space. We explicitly describe the simple objects in this category, and we show its linearized Grothendieck group is isomorphic to a natural submodule of Lusztig's periodic Hecke module. We then provide a categorification of these results by showing that the Kazhdan-Laumon Category $\mathcal{O}$ is equivalent to a full subcategory of a suitably-defined category of perverse sheaves on the semi-infinite flag variety.
\end{abstract}

\maketitle

\setcounter{tocdepth}{1}
\tableofcontents

\section{Introduction}

In \cite{KL}, with the aim of providing a new geometric construction of representations of finite Chevalley groups, D. Kazhdan and G. Laumon described a gluing construction for perverse sheaves on the basic affine space associated to a semisimple algebraic group $\mathbf{G}$ split over a finite field or its algebraic closure, defining a ``glued category" $\mathcal{A}$. This category was studied in detail (in the setting of $D$-modules rather than perverse sheaves) by R. Bezrukavnikov and A. Polishchuk in \cite{P}, and shown to be equivalent to the category of modules over global differential operators on the basic affine space in \cite{BBP}, a connection which was further explored in \cite{LS}.

In this paper, we study a subcategory $\mathcal{A}_{B}$ (and its mixed analogue $\mathcal{A}_{\mathcal{P}}$) of $\mathcal{A}$ corresponding to $B$-equivariant perverse sheaves on the flag variety; this is a natural analogue of Category $\mathcal{O}$ in the context of Kazhdan and Laumon's construction. We begin by studying this category in its own right and describing its simple objects explicitly. To state the following theorem, for any element $w$ in a Weyl group $W$ let $P(w)$ be the standard parabolic subgroup of $W$ generated by all simple reflections $s$ with $\ell(ws) > \ell(w)$.
\begin{theorem}
The category $\mathcal{A}_B$ (which we call the Kazhdan-Laumon Category $\mathcal{O}$) associated to the group $\mathbf{G}$ with Weyl group $W$ has
\begin{equation}
\sum_{w \in W} |P(w)\backslash W|
\end{equation}
irreducible objects, described explicitly in Theorem \ref{thm:simpleobj}. For any choice of $w \in W$ and coset $P(w)z$, the corresponding simple object can be written as a tuple $(\mathcal{G}_y)_{y \in W}$ of irreducible objects in $\mathrm{Perv}_B(G/B)$ such that
$$
\mathcal{G}_y \cong \begin{cases}
\mathrm{IC}_w & y \in P(w)z\\
0 & y \not\in P(w)z
\end{cases}.$$
\end{theorem}

Objects of $\mathcal{A}$ are tuples of perverse sheaves indexed by $W$, and so translation of the indices yields a natural action of $W$ on $\mathcal{A}$ (and therefore also on $\mathcal{A}_B$ and $\mathcal{A}_{\mathcal{P}}$) by functors $\mathcal{F}_w$ for $w \in W$. Further, we show that $\mathcal{A}_{\mathcal{P}}$ admits a left action by convolution with mixed $B$-equivariant perverse sheaves on the flag variety $G/B$, yielding an action of the Hecke algebra $\mathcal{H}$ on $K_0(\mathcal{A}_\mathcal{P})$. These actions give $K_0(\mathcal{A}_{\mathcal{P}}) \otimes \mathbb{C}$ the structure of a $\mathcal{H}_q \otimes \mathbb{C}[W]$-module, where $\mathcal{H}_q$ is the specialization of the generic Hecke algebra $\mathcal{H}$ over $\mathbb{C}[v, v^{-1}]$ at $v = q^{\frac{1}{2}}$. We then describe this module explicitly in terms of a more familiar combinatorial object, namely G. Lusztig's periodic Hecke module defined in \cite{L}.

\begin{theorem}\label{thm:introklomd}
There is an isomorphism $K_0(\mathcal{A}_{\mathcal{P}})\otimes \mathbb{C} \cong M_{d,q}^0$ of $\mathcal{H}_q \otimes \mathbb{C}[W]$-modules, where $M_{d,q}^0$ is a submodule (introduced in Definition \ref{def:m0d}) of Lusztig's periodic Hecke module from \cite{L}.
\end{theorem} 

In \cite{BK}, A. Braverman and Kazhdan defined the Schwartz space $\mathcal{S}$ of the basic affine space associated to an algebraic group over a non-Archimedean local field, equipped with a $\mathbb{C}[W]$-action by Fourier transform operators. They, too, found a combinatorial connection between $\mathcal{S}$ and Lusztig's module $M_{d,q}$, proving in loc.\ cit. that the subspace $\bks$ of Iwahori-invariants in $\mathcal{S}$ is isomorphic to $M_{d,q}$ as a $\mathcal{H}_q \otimes \mathbb{C}[W]$-module. Combining their result with our Theorem \ref{thm:introklomd} yields a description of $K_0(\mathcal{A}_{\mathcal{P}}) \otimes \mathbb{C}$ as a certain natural subspace $\bks_0$ of $\bks$.

\begin{theorem}\label{thm:schwartzconnection}
There is an isomorphism $K_0(\mathcal{A}_{\mathcal{P}}) \otimes \mathbb{C} \cong \bks_0$ of $\mathcal{H}_q \otimes \mathbb{C}[W]$-modules, obtained by composing the isomorphism from Theorem 1.2 with the isomorphism $ M_{d,q} \to \bks$ from \cite{BK}. This isomorphism can be described directly in terms of Grothendieck's sheaf-function dictionary: the diagram
$$\begin{tikzcd}
K_0(\mathcal{A}_\mathcal{P}) \otimes \mathbb{C} \arrow[rr] \arrow[dr] & &  \mathcal{S}^{I \times \mathbf{T}(\mathcal{O})}\\
& M_{d,q} \arrow[ur] &
\end{tikzcd}$$
commutes, where the left and right maps are the morphisms from Theorem \ref{thm:introklomd} and \cite{BK} respectively, and the horizontal map is described in Section \ref{sec:eisenstein} in terms of lifts to the local field setting of functions obtained by taking the trace of a Frobenius endomorphism on objects of $\mathcal{A}_\mathcal{P}$.
\end{theorem}

In \cite{ABBGM}, building off of work in \cite{Flags1}, \cite{FFKM2}, and \cite{BFGM}, the authors define a category of perverse sheaves on the semi-infinite flag variety associated to $\mathbf{G}$. In 6.1.8 of \cite{ABBGM}, the authors discuss how the category of Iwahori-equivariant objects of their category can be viewed as a categorification of $\bks$. Using their main result  which relates Iwahori-monodromic objects to graded modules over the small quantum group, they define a $W$-action via functors $\mathsf{F}_w$ which categorify the Fourier transforms appearing in \cite{BK}. With this in mind, we conclude this paper by upgrading Theorem \ref{thm:schwartzconnection} to an equivalence of categories.

\begin{theorem}\label{thm:introcategorification}
There exists a certain subcategory $\tilde{\mathcal{P}}$ of Iwahori-equivariant perverse sheaves on the semi-infinite flag variety with $K_0(\tilde{\mathcal{P}})\otimes \mathbb{C} \cong \bks_0$. There is an equivalence of categories between $\mathcal{A}_{\mathcal{P}}$ and $\tilde{\mathcal{P}}$ which categorifies the isomorphism $K_0(\mathcal{A}_{\mathcal{P}})\otimes \mathbb{C} \cong \bks_0$. 

This equivalence is compatible with convolution by perverse sheaves on $G/B$, which categorifies the $\mathcal{H}_q$-action on each side, and it intertwines the $W$-action by functors $\mathcal{F}_w$ on $\mathcal{A}_{\mathcal{P}}$ and $\mathsf{F}_w$ on $\tilde{\mathcal{P}}$.
\end{theorem}

We hope that Theorem \ref{thm:introcategorification} will provide a new perspective related to one of the goals stated in \cite{ABBGM} to provide a geometric description of Braverman-Kazhdan's Fourier transforms. In loc.\ cit., Fourier transform functors for perverse sheaves on the semi-infinite flag variety are only defined in the case of Iwahori-monodromic objects, as these functors come from the equivalence to modules over the small quantum group. If suitable functors could be defined on the full category of perverse sheaves on the semi-infinite flag variety, we hope that Theorem \ref{thm:introcategorification} might generalize to a suitable equivalence of categories on the full Kazhdan-Laumon category $\mathcal{A}$, thereby elevating Kazhdan and Laumon's construction from a tool intended for the study of finite Chevalley groups to an object with interesting connections to current objects of study related to a local geometric Langlands correspondence.

\subsection*{Acknowledgments} I am grateful to my advisor, Roman Bezrukavnikov, for suggesting Kazhdan-Laumon's construction as a fruitful source of new directions and providing crucial feedback and support. I would also like to thank Alexander Braverman, George Lusztig, Alexander Polishchuk, and Zhiwei Yun for helpful conversations, and Oron Propp, Alex Karapetyan, and Siddharth Mahendraker for their careful reading of the text and many helpful comments. I also acknowledge the contributions of an anonymous referee who made suggestions which I believe greatly improved and clarified parts of the paper. Finally, I would like to thank Pablo Boixeda Alvarez for pointing out an error in the definition of the map $\eta$ in an earlier version of this paper, as well as for many useful insights related to this project. During this work, I was supported by an NSERC PGS-D award.

\section{Preliminaries}

Let $\mathbf{G}$ be a semisimple algebraic group over $\mathbb{Z}$. Letting $\kappa = \mathbb{F}_q$ where $q = p^n$ for $p$ an odd prime, we write $G = \mathbf{G}_{\kappa}$ for the base change to $\mathbb{F}_q$. We assume that $G$ is split, connected, and simply-connected. Let $\mathbf{T}$ be a Cartan subgroup, $\mathbf{B}$ a Borel subgroup containing $\mathbf{T}$, and $\mathbf{U}$ its unipotent radical. We use similar notation $T, B,$ and $U$ for their respective base changes to $\kappa$. Let $X = G/B$ be the flag variety associated to $G$, considered as a variety over $\mathbb{F}_q$. Let $W$ be the Weyl group, and let $\tilde{W}$ be the affine Weyl group associated to $\mathbf{G}$, with $S \subset W$ and $\tilde{S} \subset \tilde{W}$ their respective subsets of simple reflections. We denote by $w_0$ the longest element of $W$. In Section \ref{sec:schwartz}, we will consider a local field $k$ with residue field $\kappa$. In this setting, we will denote by $G_k$ the base change of $\mathbf{G}$ to $k$, and by $X_k$ the basic affine space over $k$. Let $\Pi$ be the set of simple roots, let $\Gamma$ (resp. $\Gamma^\vee$) be the coroot (resp. weight) lattice of $\mathbf{G}$.

Let $\mathcal{H}$ be the usual Hecke algebra over $\mathbb{Z}[v, v^{-1}]$ generated by $\{T_w\}_{w \in W}$, normalized so that the quadratic relation satisfied by any $T_s$ for $s \in S$ takes the form
\begin{align}
T_s^2 & = (v^2 - 1)T_s + v^2.
\end{align}

For any $w \in W$, let $\tilde{T}_w = v^{-\ell(w)}T_w$, so that for any $s \in S$,
\begin{align}\label{eqn:heckerelation}
(\tilde{T}_s - v)(\tilde{T}_s + v^{-1}) & = 0.
\end{align}
Following the conventions of Section 3.5 of \cite{W}, recall the canonical basis $C_w$ of $\mathcal{H}$ given in terms of the Kazhdan-Lusztig polynomials $P_{x,w}$ by
\begin{align*}
C_w & = \sum_{x \leq w} (-1)^{\ell(w) - \ell(x)} v^{\ell(x)-\ell(w)}P_{x,w} \tilde{T}_{x^{-1}}^{-1}.
\end{align*}

We will sometimes consider $\mathcal{H}$ as a subalgebra of $\tilde{\mathcal{H}}$, where $\tilde{\mathcal{H}}$ is the affine Hecke algebra generated by $\{\tilde{T}_{\tilde{w}}\}_{\tilde{w} \in W}$ with $\tilde{T}_{s}$ satisfying the same relations as in (\ref{eqn:heckerelation}) for any $s \in \tilde{S}$. Choosing once and for all a square root $q^{\frac{1}{2}}$ of $q$, we let $\mathcal{H}_q$ (resp. $\tilde{\mathcal{H}}_q$) be the specialization of $\mathcal{H}$ (resp. $\tilde{\mathcal{H}}$) at $v = q^{\frac{1}{2}}$.

Throughout the paper, when $\mathcal{C}$ is an abelian or triangulated category we will denote by $K_0(\mathcal{C})$ the Grothendieck group. Suppose that $\mathcal{C}$ is triangulated equipped with a bounded $t$-structure with heart $\mathcal{C}^{\heartsuit}$ and $F : \mathcal{C} \to \mathcal{C}$ is a right or left $t$-exact triangulated functor with truncation $F^\heartsuit : \mathcal{C}^\heartsuit \to \mathcal{C}^\heartsuit$. In this case, we will sometimes refer to ``the action of $F^{\heartsuit}$ on $K_0(\mathcal{C}^\heartsuit)$" as shorthand for the natural action of $F$ on $K_0(\mathcal{C})$. We also recall the $B$-equivariant constructible derived category $D^b_{\mathrm{m}}(B\backslash G/B) = D^b_{B,\mathrm{m}}(X) = D^b_{B,\mathrm{m}}(X, \overline{\mathbb{Q}}_{\ell})$ (for $\ell$ a prime number not equal to $p$) of mixed $\ell$-adic sheaves on $G/B$, with heart $\mathrm{Perv}_{B, \mathrm{m}}(G/B) = \mathrm{Perv}_{B, \mathrm{m}}(G/B, \overline{\mathbb{Q}}_{\ell})$ under the perverse $t$-structure.

In Section \ref{sec:kazhdanlaumon} and beyond when working with perverse sheaves on the flag variety $X = G/B$, we will follow the setup of Chapter 7 of \cite{Achar}. Recalling our choice of a square root $q^{\frac{1}{2}}$ of $q$, we define the half-integer Tate twist $(\tfrac{1}{2})$ on $D^b_{\mathrm{m}}(B \backslash G/B)$. 
We then view $K_0(D^b_{\mathrm{m}}(B \backslash G/B))$ as a $\mathbb{Z}[v, v^{-1}]$-module where $v^{-1}$ acts by $(\tfrac{1}{2})$. When $\mathcal{C}$ is any category for which $K_0(\mathcal{C})$ is a $\mathbb{Z}[v, v^{-1}]$-module, we denote by $K_0(\mathcal{C})\otimes \mathbb{C}$ the specialization $K_0(\mathcal{C}) \otimes_{\mathbb{Z}[v, v^{-1}]} \mathbb{C}$ at $v = q^{\frac{1}{2}}$. We use $\mathbb{D}$ to denote the Verdier duality functor and ${}^pH^i$ to denote the $i$th perverse cohomology functor.

If $Y_1$ and $Y_2$ are varieties and $F : D_{\mathrm{m}}^b(Y_1, \overline{\mathbb{Q}}_{\ell}) \to D_{\mathrm{m}}^b(Y_2, \overline{\mathbb{Q}}_{\ell})$ is a left or right $t$-exact triangulated functor for the perverse $t$-structure, we let ${}^pF$ denote its perverse truncation. We will often apply this in the context of Grothendieck's six-functor formalism when $F = f_!$ or $F = f_*$ for a morphism $f$ of varieties. When $F$ is $t$-exact for the perverse $t$-structure, we omit the superscript and identify ${}^pF = F$ as an exact functor $\mathrm{Perv}_{\mathrm{m}}(Y_1, \overline{\mathbb{Q}}_{\ell}) \to \mathrm{Perv}_{\mathrm{m}}(Y_2, \overline{\mathbb{Q}}_{\ell})$. 

\section{Lusztig's periodic Hecke module}
In this section, we recall the periodic Hecke module considered in \cite{L1} and \cite{L}, following the notational conventions of the latter. Our main object of study is the related module $M_d$ defined in loc.\ cit. The module $M_d$ and its specialization $M_{d,q}$ will appear in every subsequent section, serving as the combinatorial data which underlies all three of the main objects of study in this paper (Kazhdan-Laumon Category $\mathcal{O}$, the Schwartz space of the basic affine space, and perverse sheaves on the semi-infinite flag variety).

\subsection{Background and setup} To begin, we follow the setup of Section 4.2 of \cite{BK}. Let $\Xi$ denote the collection of all alcoves for the group $\tilde{W}$ in the real Lie algebra $\mathfrak{t}_{\mathbb{R}}$ of $\mathbf{T}$. The set $\Xi$ admits two commuting actions of the group $\tilde{W}$, one on the left and the other on the right. We will follow the conventions of \cite{L} so that the left action of any $s \in \tilde{S}$ takes an alcove $A$ to some neighboring alcove $sA$, with the right action defined so that each $s \in \tilde{S}$ reflects an alcove around the corresponding affine hyperplane $H_s$ in $\mathfrak{t}_{\mathbb{R}}$, so that $A$ and $As$ are not, in general, neighboring. Let $\mathcal{C}^+ \subset \mathfrak{t}_{\mathbb{R}}$ denote the dominant Weyl chamber. For any $\gamma \in \Gamma^\vee$, let $A_\gamma^+$ denote the unique alcove $A \in \gamma + \mathcal{C}^+$ for which $\gamma \in \overline{A}$, and let $A^+ = A_0^+$, the ``fundamental alcove." We will also write $A_w = wA^+$ for any $w \in \tilde{W}$.

We also denote by $d : \Xi\times \Xi \to \mathbb{Z}$ (and $d_{\alpha}$ for $\alpha \in \Pi$) the ``relative distance" functions defined in \cite{L}, and we use also the definition of $\mathcal{L}(A)$ from 1.2 of loc.\ cit.; informally, for $A \in \Xi$, $\mathcal{L}(A)$ is the subset of $\tilde{S}$ for which $A$ is ``above" $sA$, where the direction is determined by that of $\mathcal{C}^+$.

\begin{defn}
The module $M_c$ is the $\mathbb{Z}[v, v^{-1}]$-span of $\Xi$ with an action of $\aH$ defined for $s \in\tilde{S}$ by
\begin{align}
\label{eqn:heckeaction}
    \tilde{T}_{s}(A) & = \begin{cases}
    sA & \text{if $s \not\in \mathcal{L}(A)$},\\
    sA  + (v - v^{-1})A & \text{if $s \in\mathcal{L}(A)$},
    \end{cases}
\end{align}
and extended naturally into a left action. It also has an action of $\Gamma$ commuting with the $\widetilde{\mathcal{H}}$-action defined by $\gamma \cdot A = A + \gamma$ for $\gamma \in \Gamma$.
\end{defn}

\begin{defn}
Let $M_{\geq}$ be the ``upward semi-infinite completion" of $M_c$, consisting of all formal sums $\sum_{A \in\Xi} m_A A$ such that the set $\{A \in \Xi~|~m_A \neq 0\}$ is bounded below (in the sense of 4.13 of \cite{L}).
\end{defn}

\subsection{An action of $W$}

For any $\alpha \in \Pi$, an operator $\theta_{\alpha} : M_\geq \to M_\geq$ is defined in \cite{L} as follows. First any $A \in \Xi$, $\alpha \in \Pi$, a sequence $A^n$, $n \geq 0$ is defined by the conditions that $A^0 = As_{\alpha}$, and $d_\alpha(A^n, A^{n+1}) > 0$, in addition to the condition that $A^n$ lie in the same ``$\alpha$-strip" as $A$ (c.f.\ \cite{L} for a more precise definition). Then
\begin{align}\label{eqn:deftheta}
    \theta_{\alpha}(A) & = v^{-1}A^0 + \sum_{n=1}^\infty(-1)^{n+1}(v^{-n+1} - v^{-n-1})A^n.
\end{align}
We define for $w \in W$ the operator $\theta_w = \theta_{\alpha_1} \cdots\theta_{\alpha_k}$ for $w = s_{\alpha_1} \cdots s_{\alpha_k}$ a reduced expression, which yields a well-defined action of the Weyl group $W$ on $M_\geq$.

\subsection{The modules $M_d$ and $M_d^0$}

In \cite{L}, Lusztig defines a duality operator $\tilde{b}$ on $M_{\geq}$ and exhibits a ``canonical basis" for $M_{\geq}$ invariant under this operator.

\begin{prop}[Theorem 12.2 in \cite{L}]
    For each $A \in \Xi$, there exists an associated element $A^\sharp \in M_{\geq}$ which satisfies $A^\sharp = A + v^{-1}\sum_{B} \mathbb{Z}[v^{-1}]B$ and $\tilde{b}(A^\sharp) = A^\sharp$.
\end{prop}

\begin{defn}
The module $M_d \subset M_{\geq}$ is the $\mathbb{Z}[v, v^{-1}]$ span of $\{A^\sharp\}_{A \in \Xi}$. It is invariant under the natural actions of $\aH$ and $\Gamma$ obtained by extending the actions defined on $M_c$. Let $\uphm{q} = M_d \otimes_{\mathbb{Z}[v, v^{-1}]} \mathbb{C}$ be the specialization of $M_d$ at $v = q^{\frac{1}{2}}$.
\end{defn}

We now introduce Lusztig's $*$-action of $W$ on $\Xi$, which we will then use to describe the action of the $\theta_w$ on $M_d$.
\begin{defn}[\cite{L1}]
Let $\mathcal{F}^*$ be the set of $\Gamma^\vee$-translates of the hyperplanes in the boundary of $\mathcal{C}^+$. 
For any $\gamma \in \Gamma^\vee$, let $\Pi_\gamma$ be the unique connected component of $\mathfrak{t}_{\mathbb{R}} \setminus \cup_{H \in \mathcal{F}^*} H$ containing $A_{\gamma}^+$.

For any $s_\alpha \in S$ and $A \in \Xi$, we then define an alcove $s_{\alpha} * A$ as follows.
First, choose the unique $\gamma \in \Gamma^\vee$ such that $A \in \Pi_\gamma$, and choose $w \in W$ so that $A = wA_{\gamma}^+$. Then define $s_{\alpha} * A = wA_{s_\alpha(\gamma)}^+$. We then define $w * A$ for any $w \in W$ by induction on $\ell(w)$. 
\end{defn}

\begin{proposition}
\label{prop:actionsharps}
For any $s_\alpha \in S$,
\begin{align}
    \theta_{s_\alpha}(A^\sharp) & = (s_\alpha * A)^\sharp,\label{eqn:thetasharp}
\end{align}
\end{proposition}
\begin{proof}
Recall the elements $A^\flat \in M_d$ defined in \cite{L}, noting that in our setup they agree with the elements $\tilde{D}_A$ defined in 8.9 of \cite{L1}. Additionally, recall the bilinear pairing $(-,-)$ on $M_d$ defined in Section 9 of \cite{L}. Corollary 8.9 of \cite{L1} gives that for any $s_\alpha \in S$, $A \in \Xi$,  $\theta_{s_\alpha}(A^\flat) = (s_\alpha * A)^\flat$. The elements $A^\sharp$ are defined such that $(A^\flat, B^\sharp) = \delta_{A,B}$. One can check directly that $(\theta_{s_\alpha}A, B) = (A, \theta_{s_\alpha}B)$ for any alcoves $A$ and $B$. So the continuity property in 3.3 of \cite{L} and the well-definedness of the $\theta_{s_\alpha}$ on $A^\sharp$ guaranteed by 12.2 of loc.\ cit. means $(\theta_sA^\flat, B^\sharp) = (A^\flat, \theta_sB^\sharp)$ for any alcoves $A, B$. This means
\begin{align*}
    (A^\flat, \theta_{s_\alpha}(B^\sharp)) & = (\theta_{s_\alpha}(A^\flat), B^\sharp)\\
    & = ((A * s_\alpha)^{\flat}, B^\sharp)\\
    & = \delta_{s_\alpha * A, B}\\
    & = \delta_{A, s_\alpha * B}.
\end{align*}
Since $\theta_{s_\alpha}(B^\sharp)$ satisfies the other defining properties in 12.2 of \cite{L}, $\theta_{s_\alpha}(B^\sharp) = (s_\alpha * B)^\sharp$, yielding (\ref{eqn:thetasharp}).
\end{proof}

\begin{defn}\label{def:epsilonindices}
For any $z \in W$ and $\tilde{w} \in \aW$, let $\epsilon_z(\tilde{w})$ be the element of $\aW$ such that $z * A_{\tilde{w}} = A_{\epsilon_z(\tilde{w})}$.
\end{defn}

\begin{defn}\label{def:m0d}
Let $M_d^0$ be the $\mathbb{Z}[v, v^{-1}]$-submodule of $M_d$ generated by the finitely-many elements $\{A_{w}\}_{w \in W}$ and their images under the operators $\{\theta_w\}_{w \in W}$. Let $M_{d,q}^0$ be the specialization of $M_{d}^0$ at $v = q^{\frac{1}{2}}$.
\end{defn}

\begin{lemma}
$M_{d,q}^0$ is a finite-dimensional $\mathcal{H}_{q}\otimes \mathbb{C}[W]$-submodule of $M_{d,q}$.
\end{lemma}
\begin{proof}
The fact that $M_{d,q}^0$ is a finite-dimensional $\mathbb{C}[W]$-submodule follows by definition. Since for any $w \in W \subset \tilde{W}$ and any $y \in W$ we have $\tilde{T}_yA_w \in \mathrm{span}\{A_z\}_{z \in W}$ by the definition of the $\mathcal{H}_q$-action. Since this action commutes with the action of $\mathbb{C}[W]$, it follows that $M_{d,q}^0$ is closed under the action of $\mathcal{H}_q$.
\end{proof}

\subsection{Example in Type $\mathbf{A}_1$}\label{sec:examplesl2}

\label{sec:sl2phm}
For $\mathbf{G} = \mathbf{SL}_2$, we can identify $\Xi$ with the set of intervals $(n, n+1) \subset \mathbb{R}$ for $n \in \mathbb{Z}$, which we denote by $A_n$. With this convention, $A_{e} = A_{0}$ and $A_{s_1} = A_{-1}$. It follows from the definitions in \cite{L} that
\begin{align}
    s_1 * A_n & = A_{-n}\\
    A_n^{\sharp} & = A_n + \sum_{i=1}^\infty (-1)^i v^{-i}A_{n+i} \in M_d.\label{eqn:sl2simple}
\end{align}

A computation then shows that
\begin{align*}
    \tilde{T}_{s_1}A_n & = \begin{cases}
    A_{n+1} & \text{if $n$ is odd},\\
    A_{n-1} + (v - v^{-1})A_{n} & \text{if $n$ is even,}
    \end{cases}\\
    \tilde{T}_{s_1}A_n^{\sharp} & = \begin{cases}
    -v^{-1}A_{n}^\sharp & \text{if $n$ is odd,}\\
    vA_n^\sharp + A_{n-1}^\sharp + A_{n+1}^\sharp & \text{if $n$ is even.}
    \end{cases}
\end{align*}

Further,
\begin{align*}
    \theta_{s_1}(A_n^\sharp) & = A_{-n}^\sharp.\\
\end{align*}

\section{Kazhdan-Laumon Category $\mathcal{O}$\label{sec:kazhdanlaumon}}

In this section, we recall the Kazhdan-Laumon construction for gluing perverse sheaves on the basic affine space described in \cite{KL}. We define a subcategory $\mathcal{A}_B$ of the Kazhdan-Laumon category $\mathcal{A}$ analogous to Category $\mathcal{O}$ along with a mixed version $\mathcal{A}_{\mathcal{P}}$, and describe its simple objects explicitly.

\subsection{Category $\mathcal{O}$, mixed sheaves, and the Hecke algebra\label{sec:cato}}

Recall that $X$ is stratified by affine cells $\{X_w\}_{w \in W}$, and for each $w \in W$ we let $j_w : X_w \hookrightarrow X$ be the inclusion. We will use the following definitions of standard, costandard, and simple perverse sheaves in $\mathrm{Perv}_{B,\mathrm{m}}(X, \Ql)$ with Tate twists (the notation in (\ref{eqn:stdcostdsimple}) which we will use throughout the paper differs by a Tate twist from that of \cite{Achar}, but will be more convenient for our purposes).
\begin{equation}\label{eqn:stdcostdsimple}
\begin{split}
\Delta_w & = j_{w!}(\Qlu[\ell(w)])(\tfrac{\ell(w)}{2})\\
\nabla_w & = j_{w*}(\Qlu[\ell(w)])(\tfrac{\ell(w)}{2})\\
\mathrm{IC}_w & = j_{w!*}(\Qlu[\ell(w)])(\tfrac{\ell(w)}{2})
\end{split}
\end{equation}
As in 7.2 of \cite{Achar}, we also have the convolution product
\begin{align}
    * : D^b_{\mathrm{m}}(B\backslash G/B, \Ql) \times D^b_{\mathrm{m}}(B\backslash G/B, \Ql) & \to D^b_{\mathrm{m}}(B \backslash G/B, \Ql).\label{eqn:ringconv}
\end{align}

\begin{defn}
Let $\mathcal{P}$ be the full subcategory of $\mathrm{Perv}_{B,\mathrm{m}}(X, \Ql)$ generated by objects of the form $\mathrm{IC}_w(\tfrac{m}{2})$ for $w\in W$, $m \in \mathbb{Z}$.
\end{defn}

The convolution product in (\ref{eqn:ringconv}) gives a ring structure on $K_0(\mathcal{P})$, which we will identify with the Grothendieck group of the subcategory of $D_{\mathrm{m}}^b(B\backslash G/B)$ generated  as a triangulated category by twists and shifts of the objects $\mathrm{IC}_w$, since $\mathcal{P}$ is the heart of the latter category.

\begin{prop}[c.f.\ Ex. 7.4.3 in \cite{Achar}]
There is a unique isomorphism of rings
\begin{align*}
\mathrm{ch} : K_0(\mathcal{P}) \cong \mathcal{H}
\end{align*}
such that $\mathrm{ch}([\nabla_w]) = \tilde{T}_w$ for all $w \in W$. Moreover, this map satisfies
\begin{equation}
\label{eqn:mixed}
\begin{aligned}
\mathrm{ch}([\mathcal{F}(\tfrac{1}{2})]) & = v^{-1}\mathrm{ch}([\mathcal{F}]), & \mathrm{ch}([\mathbb{D}(\mathcal{F})]) & = \overline{\mathrm{ch}([\mathcal{F}])},\\
\mathrm{ch}([\mathrm{IC}_w])& = C_w, & \mathrm{ch}([\nabla_w(-\tfrac{\ell(w)}{2})]) & = T_w.
\end{aligned}
\end{equation}
\end{prop}
This also implies that $K_0(\mathcal{P})$ is isomorphic to $\mathcal{H}$ as a bimodule over itself via the left and right convolution actions of $K_0(\mathcal{P})$.

\subsection{Kazhdan-Laumon categories}\label{sec:klcats}

In \cite{KL}, the authors construct an abelian category $\mathcal{A}$ by ``gluing" $|W|$-many copies of the category $\mathrm{Perv}(G/U)$ via Fourier transforms ${}^p\tilde{F}_{s,!}$ (called ${}^pF_{s,!}$ in loc.\ cit.) indexed by simple reflections $s \in S$. These come from right $t$-exact functors $\tilde{F}_{s,!} : D^b(G/U) \to D^b(G/U)$. In \cite[5.2]{P}, it is shown that
\begin{align}\label{eqn:defwithks}
    \tilde{F}_{s,!}(\mathcal{G}) = p_{2!}(p_1^*\mathcal{G} \otimes \overline{K(s)}),
\end{align}
where $p_1, p_2$ are the projection maps from $G/U \times G/U \to G/U$, and $\overline{K(s)}$ is defined as follows.

In loc.\ cit., a collection of subvarieties of $G/U \times G/U$ indexed by the Weyl group is introduced. For $s \in S$, we describe the subvareity $X(s)$ indexed by $s$ explicitly. Write
\begin{align}
    n_s = \rho_s\begin{pmatrix}
        0 & 1\\ 
        -1 & 0
    \end{pmatrix}
\end{align}
where $\rho_s : \mathrm{SL}_2 \to G$ is the homomoprhism described in in loc.\ cit. corresponding to the simple root $\alpha_s$, and let $T_s = \alpha_s^{\vee}(\mathbb{G}_m) \subset T$. Then $X(s)$ is the subvariety of pairs $(gU, g'U) \in G/U\times G/U$ such that $g^{-1}g' \in Un_sT_sU$. There is a morphism $\mathrm{pr}_s : X(s) \to T_s \cong \mathbb{G}_m$ sending $(gU, g'U)$ to the $t_s \in T_s$ such that $g^{-1}g' = Un_st_sU$, and this extends to a morphism $\overline{\mathrm{pr}}_s : \overline{X(s)} \to \mathbb{G}_a$. The sheaf $\overline{K(s)}$ on $G/U \times G/U$ is then defined by $\overline{K(s)} = (-\overline{\mathrm{pr}}_s)^*\mathcal{L}_{\psi}$, where $\mathcal{L}_{\psi}$ is the Artin-Schreier sheaf on $\mathbb{G}_a$ corresponding to $\psi$.

It is shown in \cite{KL} that the functors $\{\tilde{F}_{s,!}\}_{s\in S}$ assemble to give an action of the generalized braid group $B_W$ on $D^b(G/U)$; this defines a functor $\tilde{F}_{w,!}$ associated to every $w \in W$ by composing the functors corresponding to the simple reflections in any reduced word for $w$. Objects of $\mathcal{A}$ are then tuples indexed by $W$ of elements of $\mathrm{Perv}(G/U)$ equipped with the additional structure explained in \cite{P}. For each $w \in W$, there is an exact functor $j_w^* : \mathcal{A} \to \mathrm{Perv}(G/U)$ defined by $j_w^*((\mathcal{G}_{w'})_{w' \in W}) = \mathcal{G}_w$. Each of the functors $j_w^*$ has a left adjoint $j_{w!}$ and a right adjoint $j_{w*}$, each of which is a functor from $\mathrm{Perv}(G/U) \to \mathcal{A}$. For a simple object $\mathcal{G}$ of $\mathrm{Perv}(G/U)$ and a choice of $w \in W$ a simple object $j_{w!*}(\mathcal{G})$ is defined in loc.\ cit. as the image of the natural map $j_{w!}(\mathcal{G}) \to j_{w*}(\mathcal{G})$ obtained by adjunction. Finally, we also have the exact functors $\{\mathcal{F}_w\}_{w \in W}$, which act by $\mathcal{F}_w((\mathcal{G}_{w'})_{w' \in W}) = (\mathcal{G}_{w'w})_{w' \in W}$. In Section 9 of \cite{P}, this same gluing construction is applied to the category $\mathrm{Perv}_{\mathrm{m}}(G/U) = \mathrm{Perv}_{\mathrm{m}}(G/U, \Ql)$ of mixed perverse sheaves on $G/U$ to construct a category $\mathcal{A}_{\mathrm{m}}$.

We now note that the same construction can be applied to $\mathrm{Perv}(G/B)$ or its mixed version $\mathrm{Perv}_{\mathrm{m}}(G/B)$. Indeed, let $\pi : G/U \to G/B$ be the natural quotient morphism. Then we can define the functor $F_{w,!}$ with source $D^b(G/B)$ by setting $F_{w,!} = \tilde{F}_{w,!} \circ \pi^*[r](\tfrac{r}{2}),$ where $r$ is the rank of $T$. In the following result, we will show that anything in the image of $F_{w,!}$ is also in the image of $D^b(G/B)$ under $\pi^*[r](\tfrac{r}{2})$, so ${}^pF_{w,!}$ can be thought of as a functor from $\mathrm{Perv}(G/B)$ to itself (and similarly for its mixed analogue), and further can be easily described as a convolution or \emph{intertwining functor}. This result is a categorification of a statement which appears on the level of Grothendieck groups in \cite[6.3]{P}.

\begin{proposition}\label{prop:f!conv}
The functors ${}^pF_{w,!}$ are well-defined endofunctors of the categories $\mathrm{Perv}_{\mathrm{m}}(X)$ and $\mathcal{P}$. Indeed, for any $\mathcal{G} \in \mathrm{Perv}_{\mathrm{m}}(X)$, and $s \in S$,
\begin{align}\label{eqn:convcostd}
F_{s,!}(\mathcal{G}) & = \mathcal{G} * \nabla_s(\tfrac{1}{2}).
\end{align}
\end{proposition}
\begin{proof}
Let $\mathcal{G} \in \mathrm{Perv}_{\mathrm{m}}(X)$. We will first show that $F_{s,!}(\Delta_e) \cong \pi^*\nabla_s[r](\tfrac{r+1}{2})$ as elements of $\mathrm{Perv}_{B,\mathrm{m}}(G/U)$. Then by abuse of notation we can simply write $F_{s,!}(\Delta_e) = \nabla_s(\tfrac{1}{2})$, and we then note that it is easy to see that $F_{s,!}$ commutes with left convolution, meaning
    \begin{align*}
        F_{s,!}(\mathcal{G}) & = \mathcal{G} * F_{s,!}(\Delta_{e})\\
        & = \mathcal{G} * \nabla_s(\tfrac{1}{2}),
    \end{align*}
    and so the proof will be complete.

    Accordingly, we now compute $F_{s,!}(\Delta_e)$. Letting $j : T \to G/U$ be the natural inclusion, we then have
    \begin{align}
        F_{s,!}(\Delta_e) & = p_{2!}(p_1^*j_!{\overline{\underline{\mathbb{Q}}}_{\ell}} \otimes \overline{K(s)})[r](\tfrac{r}{2})\\
        & = p_{2!}((j\times \mathrm{id})_! {\overline{\underline{\mathbb{Q}}}_{\ell}} \otimes \overline{K(s)})[r](\tfrac{r}{2})\\
        & = p_{2!}(j \times \mathrm{id})_{!}(j\times \mathrm{id})^* \overline{K(s)}[r](\tfrac{r}{2})\\
        & = p_{2!}(j \times \mathrm{id})_{!}(j\times \mathrm{id})^* (-\overline{\mathrm{pr}}_s)^*\mathcal{L}_{\psi}[r](\tfrac{r}{2})\label{eqn:leftoff}
    \end{align}
    We now note that $(-\overline{\mathrm{pr}}_s) \circ (j\times \mathrm{id}) = \gamma_s \circ m$ where $\gamma_s : G/U \to \mathbb{G}_a$ is defined as follows. Following \cite{KL}, we note that $G/U \to G/Q_s$ is a fibration with $\mathbb{A}^2 \setminus \{(0, 0)\}$ fibers; one can then define a $G$-invariant symplectic form $\langle, \rangle_s$ defined fiberwise as in loc.\ cit. Then we define $\gamma : G/U \to \mathbb{G}_a$ by $\gamma(gU) = \langle gU, U\rangle_s$. This is zero away from $\overline{X(s)} \cong \mathbb{A}^2 \setminus \{(0, 0)\}$. Under this identification, $\gamma$ can be thought of as the projection $\mathbb{A}^2 \setminus \{(0, 0)\} \to \mathbb{G}_a$ to the second coordinate. This reduces the claim in the proposition to the $SL_2$ case where $G/U \cong \mathbb{A}^2 \setminus \{(0, 0)\}$. 

    This means to complete the proof, we must check that
    \begin{align*}
        \overline{p}_{2!}(m^*\gamma^*\mathcal{L}_{\psi}) \cong \tilde{j}_*{\overline{\underline{\mathbb{Q}}}_{\ell}}(\tfrac{1}{2})
    \end{align*}
    where $\overline{p}_2 : \mathbb{G}_m \times (\mathbb{A}^2 \setminus \{(0, 0)\}) \to (\mathbb{A}^2 \setminus \{(0, 0)\})$ is the projection, $m : \mathbb{G}_m \times (\mathbb{A}^2 \setminus \{(0, 0)\}) \to \mathbb{A}^2 \setminus \{(0, 0)\}$ is the scaling action, and $\tilde{j} : X(s) \cong (\mathbb{A}^2 \setminus \mathbb{A}^1) \hookrightarrow \mathbb{A}^2 \setminus \{(0, 0)\}$ is the inclusion of the complement of the $x$-axis.

    By Grothendieck's projection formula, this reduces further to the following simpler comparison between sheaves on $\mathbb{A}^1$. We must check that
    \begin{align}
        p_{\mathbb{G}_a!}(m^*\mathcal{L}_{\psi}) \cong j_{\mathbb{G}_m*}{\overline{\underline{\mathbb{Q}}}_{\ell}}(\tfrac{1}{2}),\label{eqn:gacomp}
    \end{align}
    where $p_{\mathbb{G}_a} : \mathbb{G}_m \times \mathbb{G}_a \to \mathbb{G}_a$ is the projection, $m : \mathbb{G}_m \times \mathbb{G}_a$ is multiplication and $j_{\mathbb{G}_m} : \mathbb{G}_m \to \mathbb{G}_a$ is the inclusion away from zero. Using the results in Section 4 of \cite[Sommes trig.]{DCoh}, one can check that $j_{\mathbb{G}_m}^*p_{\mathbb{G}_a!}(m^*\mathcal{L}_{\psi}) \cong {\overline{\underline{\mathbb{Q}}}_{\ell}}(\tfrac{1}{2})$, giving a natural morphism by adjunction from the left-hand side to the right-hand side in (\ref{eqn:gacomp}). It is then straightforward to check that this is an isomorphism on all stalks, completing the proof that the two sides of (\ref{eqn:gacomp}) are isomorphic. This then implies that $F_{s,!}(\Delta_e) = \nabla_s(\tfrac{1}{2})$, completing the proof.
\end{proof}

\begin{defn}
Let $\mathcal{A}_{B,\mathrm{m}}$ (resp. $\mathcal{A}_{B}$) be the Kazhdan-Laumon category obtained by applying the gluing procedure described in \cite{KL} and \cite{P} to $\mathrm{Perv}_{B, \mathrm{m}}(X, \Ql)$ (resp. $\mathrm{Perv}_{B}(X, \Ql)$) and the functors ${}^pF_{w,!}$. Alternatively, Theorem 1.2.1 in \cite{P} ensures that $\mathcal{A}_{B}$ is the full subcategory of $\mathcal{A}$ consisting of objects $B$-equivariant on both the left and right.

Similarly, let $\mathcal{A}_{\mathcal{P}}$ be the category obtained in this same way from $\mathcal{P}$.
\end{defn}

\begin{remark}
    In this paper, we choose to work with $\mathrm{Perv}_B(G/B)$ and $\mathcal{A}_B$ so that the left and right Hecke actions are easy to describe. One could just as easily work with $\mathrm{Perv}_U(G/B)$, which is a geometric model for a different version of the BGG Category $\mathcal{O}$. In this case, the gluing procedure and its description via Proposition \ref{prop:f!conv} still hold, as will our description of simple objects. The only difference is that the geometric model for the Hecke category which acts on the left is the free-monodromic Hecke category considered by Bezrukavnikov and Yun in \cite{BY}. For ease of exposition, we work with $\mathrm{Perv}_B(G/B)$ avoid introducing the technicalities required to define free-monodromic sheaves, but we note that there is no technical obstruction to doing so.
    
    If one does define the Kazhdan-Laumon Category $\mathcal{O}$ by gluing together the categories $\mathrm{Perv}_U(G/B)$, one will also get an analogue of our Theorem \ref{thm:introcategorification}, replacing Iwahori-equivariant sheaves on the semi-infinite flag variety with Iwahori-monodromic ones, and the proof is exactly the same.
\end{remark}

We note that both $K_0(\mathcal{P})$ and $K_0(\mathcal{A}_\mathcal{P})$ are free $\mathbb{Z}[v, v^{-1}]$-modules. Following \cite{P} let $\phi_s$ denote the action of the convolution $- * \nabla_s(\tfrac{1}{2})$ on $K_0(\mathcal{P})$. For any $s \in S$ let $Q_s = [P_s, P_s]$ (for $P_s$ the parabolic subgroup associated to $s$) and $p_s : G/U \to G/Q_s$ the projection.

\begin{prop}[c.f. \cite{P}, Theorem 5.6.1]\label{prop:psqs}
    The Grothendieck group $K_0(\mathcal{A}_{\mathcal{P}})$ (resp. $K_0(\mathcal{A}_B)$) sits naturally as a subspace of $\oplus_{w \in W} K_0(\mathrm{Perv}_{B,\mathrm{m}}(G/B))$ (resp. $\oplus_{w \in W} K_0(\mathrm{Perv}_B(G/B))$).

    In fact, we have
    \begin{align}
        K_0(\mathcal{A}_{\mathcal{P}}) & = \{(c_z)_{z \in W}~|~ \phi_sc_z - c_{sz} \in p_s^*(K_0(\mathrm{Perv}_{\mathrm{m}}(G/Q_s))), z \in W, s \in S\}\\
        & \subset \oplus_{w \in W} K_0(\mathrm{Perv}_{B,\mathrm{m}}(G/B)).
    \end{align}
    and similarly in the non-mixed case for $K_0(\mathcal{A}_B)$.
\end{prop}
\begin{proof}
    This is exactly Theorem 5.6.1 of \cite{P}, which is in turn a special case of Theorem 1.2.1 of loc.\ cit. The latter is proved by considering the derived gluing data formed from the derived functors of $j_{w!}$ and $j_w^*$ as functors on the derived categories of the underlying abelian categories. In our $B$-equivariant setup, we note that the equivariant derived category $D_{B, \mathrm{m}}^b(G/B)$ does not coincide with $D^b(\mathrm{Perv}_{B,\mathrm{m}}(G/B))$. That being said, the proof of Theorem 1.2.1 in loc.\ cit.\ goes through for any derived gluing data equipped with a $t$-structure whose underlying heart is the abelian category under consideration (so long as the same $t$-exactness properties are assumed), so the proof still carries forward to the setting in the present paper. The technicalities of derived Kazhdan-Laumon gluing are discussed in \cite[Section 4]{CMF3}, but will not used in the present paper.
\end{proof}

We now recall that $K_0(\mathcal{P})$ admits an $\mathcal{H}$-action on the left by convolution.

\begin{lemma}\label{lem:lefthaction}
There is a left action of $\mathcal{H}$ on $K_0(\mathcal{P})$ given for $s \in S$ by
\begin{align*}
\tilde{T}_s \cdot [\mathcal{G}] & = [\nabla_s * \mathcal{G}] = [{}^pH^{0}(\nabla_s * \mathcal{G})] - [{}^pH^{-1}(\nabla_s * \mathcal{G})]
\end{align*}
for any $\mathcal{G} \in \mathcal{P}$
\end{lemma}
\begin{proof}
    The only thing to check is that ${}^pH^i(\nabla_s * \mathcal{G}) = 0$ for $i \neq 0, -1$. We first note that $\nabla_s * -$ is right $t$-exact, so ${}^pH^i(\nabla_s * \mathcal{G}) = 0$ for $i > 0$. The long-exact sequence of perverse cohomology associated to the short-exact sequence
    \[\begin{tikzcd}
        0 \arrow[r] & \mathrm{IC}_s \arrow[r] & \nabla_s \arrow[r] & \Delta_e(-\tfrac{1}{2}) \arrow[r] & 0
    \end{tikzcd}\]
    shows that ${}^pH^i(\nabla_s * \mathcal{G}) \cong {}^pH^i(\mathrm{IC}_s)$ for $i < 0$, and the long-exact sequence of perverse cohomology associated to
    \[\begin{tikzcd}
        0 \arrow[r] & \Delta_e(\tfrac{1}{2}) \arrow[r] & \Delta_s \arrow[r] & \mathrm{IC}_s \arrow[r] & 0
    \end{tikzcd}\]
    along with the left $t$-exactness of $\Delta_s * -$ shows that ${}^pH^i(\mathrm{IC}_s * \mathcal{G}) = 0$ for $i < -1$. 
\end{proof}
We will use the following lemma to show that the left $\mathcal{H}$-action in Lemma \ref{lem:lefthaction} naturally lifts to $K_0(\kl)$.

\begin{lemma}
    For any $\mathcal{G} \in \mathcal{P}$ and $(\mathcal{A}_w)_{w \in W} \in \mathcal{A}_{\mathcal{P}}$, there is a natural way to define an object $({}^pH^i(\mathcal{G} * \mathcal{A}_w)_{w \in W}) \in \mathcal{A}_{\mathcal{P}}$ for any $i \in \mathbb{Z}$.
\end{lemma}
\begin{proof}
    By the definition of Kazhdan-Laumon categories in \cite{KL} and \cite{P}, we know that $(\mathcal{A}_{w})_{w \in W} \in \mathcal{A}$ is a collection of objects $\mathcal{A}_w \in \mathcal{P}$ equipped with morphisms
    \begin{align}
        \theta_{w,s} : {}^pF_{s,!}\mathcal{A}_w \to \mathcal{A}_{sw}
    \end{align}
    for every $w \in W$, $s \in S$ satisfying the compatibilities described in \cite[Section 1.1]{P}. For ease of notation, we let $G : D_{B,\mathrm{m}}^b(G/B) \to D_{B,\mathrm{m}}^b(G/B)$ be the functor defined by convolution on the left with $\mathcal{G}$, i.e. $G(\mathcal{F}) = \mathcal{G} * \mathcal{F}$. We note that there is a natural isomorphism of triangulated functors $G \circ F_{s,!} \cong F_{s,!} \circ G$ by Proposition \ref{prop:f!conv} and associativity of convolution. Letting $\tau_{\leq 0}$ and $\tau_{\geq 0}$ be the perverse truncation functors, we note that $F_{s,!}$ is right $t$-exact, and therefore we have natural isomorphisms $\tau_{\leq 0}F_{s,!}\tau_{\leq 0} \cong F_{s,!}\tau_{\leq 0}$ and $\tau_{\geq 0}F_{s,!}\tau_{\geq 0} \cong \tau_{\geq 0}F_{s,!}$.

    We will now produce for any $w \in W$, $s \in S$ a morphism
    \begin{align*}
        {}^0\theta_{w,s}^{\mathcal{G}} : {}^{p}F_{s,!}({}^pH^0(\mathcal{G} * \mathcal{A}_w)) \to {}^pH^0(\mathcal{G} * \mathcal{A}_{sw}).
    \end{align*}
    Indeed, from the above we have the sequence of morphisms and identifications
    \begin{align}
        {}^pF_{s,!}({}^pH^0(\mathcal{G} * \mathcal{A}_w)) & \cong {}^pH^0(F_{s,!} \tau_{\leq 0}G\mathcal{A}_w)\label{eq:tau2}\\
        & \to {}^pH^0(F_{s,!}G\mathcal{A}_w)\label{eq:tau3}\\
        & \cong {}^pH^0(GF_{s,!}\mathcal{A}_w)\label{eq:tau4}\\
        & \to {}^pH^0(G\tau_{\geq 0}F_{s,!}\mathcal{A}_w)\label{eq:tau5}\\
        & = {}^pH^0(\mathcal{G} * {}^pF_{s,!}\mathcal{A}_w)\label{eq:tau6}\\
        & \to {}^pH^0(\mathcal{G} * \mathcal{A}_{sw}).\label{eq:tau7}
    \end{align}
    The identifications and morphisms arise as follows. The equivalence (\ref{eq:tau2}) is by right $t$-exactness of $F_{s,!}$, (\ref{eq:tau3}) and (\ref{eq:tau5}) arise from the natural truncation morphisms $\tau_{\leq 0}A \to A$ and $A\to \tau_{\geq 0}A$ for any object $A$, and (\ref{eq:tau7}) is induced by $\theta_{w,s}$.
    
    Replacing the triangulated functor $G$ with its shift $G[i]$ and applying the same sequence of morphisms gives a similar morphism
    \begin{align}   {}^i\theta_{w,s}^{\mathcal{G}} : {}^{p}F_{s,!}({}^pH^i(\mathcal{G} * \mathcal{A}_w)) \to {}^pH^i(\mathcal{G} * \mathcal{A}_{sw})
    \end{align}
    for any $i \in \mathbb{Z}$.

    It is a straightforward check that the necessary compatibilities detailed in \cite[Section 1.1]{P} are satisfied by each of the morphisms ${}^i\theta_{w,s}^{\mathcal{G}}$ just constructed; this follows from the fact that the morphisms $\theta_{w,s}$ satisfy these compatibilities, combined with the functoriality of each morphism used in the construction of ${}^i\theta_{w,s}^{\mathcal{G}}$.
\end{proof}

\begin{proposition}\label{prop:heckeoaction}
There is a left action of $\mathcal{H}$ on $K_0(\kl)$ given for $s \in S$ by
\begin{align*}
    \tilde{T}_s \cdot [(\mathcal{A}_w)_{w \in W}] & = [{}^pH^{0}(\nabla_s * \mathcal{A}_w)_{w \in W}] - [{}^pH^{-1}(\nabla_s * \mathcal{A}_w)_{w \in W}]
\end{align*}
for any $(\mathcal{A}_w)_{w \in W} \in \mathcal{A}_\mathcal{P}$, $s \in S$.
\end{proposition}

\begin{remark}
    We note that to begin with, $D_{B,\mathrm{m}}^b(G/B)$ carries two commuting monoidal actions of itself, via convolution on the left and on the right. Proposition \ref{prop:f!conv} shows that the Kazhdan-Laumon gluing functors can be described via the perverse truncation of the right action restricted to $\mathcal{P}$. We then use this action to construct $\mathcal{A}_{\mathcal{P}}$ from $\mathcal{P}$, which ``consumes" the right action. However, the perverse truncation of the left action, which is carried by $\mathcal{P}$, is then still carried by the glued category $\mathcal{A}_{\mathcal{P}}$.

    As a result, when we make reference to the Hecke module structure on $K_0(\mathcal{A}_{\mathcal{P}})$ the left action will appear, while the right action will show up in the context of functors such as $j_{w!}$ which are inherent to the Kazhdan-Laumon gluing construction.
\end{remark}

\subsection{Simple objects in $\kl$}

We will now classify the simple objects in $\mathcal{A}_{\mathcal{P}}$ up to Tate twist, which correspond to the simple objects in $\mathcal{A}_{B}$. We first state our main result. To do so, for any $w \in W$, let $P(w)$ denote the standard parabolic subgroup of $W$ generated by the simple reflections $s \in S$ for which $\ell(ws) > \ell(w)$.
\begin{theorem}
\label{thm:simpleobj}
Up to Tate twist, any simple object in $\kl$ is of the form $j_{z!*}(\mathrm{IC}_w)$ for some $w, z \in W$. Two such objects $j_{z_1!*}(\mathrm{IC}_{w_1})$ and $j_{z_2!*}(\mathrm{IC}_{w_2})$ are isomorphic if and only if $w_1 = w_2$ and $z_2 \in P(w_1)z_1$. 

Accordingly, simple objects in $\kl$ up to Tate twist are in bijection with pairs $(w, \overline{w'})$, where $w \in W$, and $\overline{w'}$ is an element of $P(w) \backslash W$.
\end{theorem}

To prove Theorem \ref{thm:simpleobj}, we will need an intermediate result, Proposition \ref{prop:simples}. In the following, we freely use Lemma 1.3.1 in \cite{P}, which guarantees that if $\mathcal{G}\in \mathcal{A}$ is simple, then for all $w \in W$, $j_w^*(\mathcal{G})$ is either zero or simple in $\mathrm{Perv}(X)$, and $\mathcal{G} \cong j_{w!*}(j_w^*(\mathcal{G}))$ for any $w$ for which it is nonzero. We begin with a lemma about Kazhdan-Lusztig polynomials $P_{y,w}$ which we will need for our result. This is well-known; one concise proof is given in 4.3.2(ii) of \cite{W}.

\begin{lemma}\label{lem:klpolys}
If $w \in W$ and $s$ is a simple reflection such that $\ell(ws) < \ell(w)$, then $P_{ys,w} = P_{y,w}$ for all $y < w$. 
\end{lemma}

\begin{proposition}\label{prop:simples}
If $\mathcal{G} \in \kl$ is simple and $j_z^*\mathcal{G} = \mathrm{IC}_w$ for some $z \in W$ (in other words, $\mathcal{G}  \cong j_{z!*}\mathrm{IC}_w$), then for any $y \in W$,
\begin{align*}
   j_{yz}^*\mathcal{G} \cong \begin{cases}
   \mathrm{IC}_w & \text{if } y \in P(w)\\
   0 & \text{if } y \not\in P(w).
\end{cases}
\end{align*}
\end{proposition}
\begin{proof}
Suppose $\mathcal{G} = j_{z!*}\mathrm{IC}_w$. Pick some $s \in P(w)$, and following \cite{P} let $\phi_s$ denote the action of the convolution $- * \nabla_s(\tfrac{1}{2})$ on $K_0(\mathcal{P})$. By Lemma 1.3.1 in loc.\ cit., $j_a^*\mathcal{G} \in \mathcal{P}$ is either simple or zero for every $a \in W$. Writing $[\mathcal{\mathcal{G}}] = (c_a)_{a \in W}$ for $c_a \in K_0(\mathcal{P})$, Proposition \ref{prop:psqs} ensures $\phi_sc_z - c_{sz} \in p_s^*(K_0(\text{Perv}(G/Q_s)))$. Further, an element of $K_0(\mathcal{P})$ is in $p_s^*(K_0(\text{Perv}(G/Q_s)))$ if and only if it is in the subspace $K_s$ of $K_0(\mathcal{P})$ spanned by Tate twists of $[\mathrm{IC}_{w'}]$ for $w'$ with $\ell(w's) < \ell(w')$.

Using Lemma \ref{lem:klpolys}, it is easy to check that if $\ell(ws) > \ell(w)$, then
\begin{align*}
\phi_s[\mathrm{IC}_w] - [\mathrm{IC}_w] \in K_s.
\end{align*}
Since $\phi_sc_z - c_{sz}$ and $\phi_s[\mathrm{IC}_w] - [\mathrm{IC}_w]$ both lie in $K_s$, it follows that $c_{sz} - [\mathrm{IC}_w] \in K_s$. Since $c_{sz}$ is the class in $K_0(\mathcal{P})$ of a simple object, we must have $c_{sz} = [\mathrm{IC}_w]$, and so $j_{sz}^*\mathcal{F} \cong \mathrm{IC}_w$. By Lemma 1.3.2 of \cite{P}, this means $j_{z!*}\mathrm{IC}_w = j_{sz!*}\mathrm{IC}_w$, and so by induction on $\ell(y)$, this shows that $j_{yz}^*\mathcal{G} \cong \mathrm{IC}_w$ whenever $y \in P(w)$. We note that this implies that $\mathcal{G} \cong j_{yz!*}\mathrm{IC}_{w}$ for any $y \in P(w)$. 

Now suppose that $y \not\in P(w)$. We can then write a reduced expression $y = y_1sy_2$ where $y_2 \in P(w)$, $y_1 \in W$ and $s$ is a simple reflection with $\ell(ws) < \ell(w)$. Letting $z' = y_2z$, by the previous paragraph, $j_{z!*}\mathrm{IC}_w \cong \mathcal{G} \cong j_{z'!*}\mathrm{IC}_{w}$, so
\begin{align}
    j_{y_1sz'}^*\mathcal{G} \cong j_{y_1sz'}^*j_{z'!*}\mathrm{IC}_w.
\end{align}

We now claim that $${}^pF_{y_1s,!}(\mathrm{IC}_w) = {}^pH^0(\mathrm{IC}_w * \nabla_{sy_1^{-1}}(\tfrac{\ell(y_1s)}{2})) = 0.$$ This will prove the result, since by the definition of $j_{z'!*}$, $j_{y_1sz'}^*(j_{z'!*}(\mathrm{IC}_w))$ is the image of $j_{y_1sz'}^*(j_{z'!}(\mathrm{IC}_w)) = {}^pF_{y_1s,!}(\mathrm{IC}_w)$ in $j_{y_1sz'}^*(j_{z'*}(\mathrm{IC}_w))$ under the natural map obtained by adjunction.

Indeed, $\mathrm{IC}_w * \nabla_s(\tfrac{1}{2}) \cong \mathrm{IC}_w(1)[1]$ (see, e.g. 7.2.5 in \cite{Achar}), and $- * \nabla_{y_1^{-1}}$ is right $t$-exact, meaning $\mathrm{IC}_w * \nabla_{sy_1^{-1}} \cong \mathrm{IC}_w(1) * \nabla_{y_1^{-1}}[1]$ can lie in perverse degrees at most $-1$. This means ${}^pH^0(\mathrm{IC}_w * \nabla_{sy_1^{-1}}(\tfrac{\ell(y_1s)}{2})) = 0$ as desired.
\end{proof}

\begin{proof}[Proof of Theorem \ref{thm:simpleobj}]
The first statement follows from Lemma 1.3.2 of \cite{P} (since twists of $\mathrm{IC}_w$ are the simple objects of $\mathcal{P}$), so it remains only to determine when simple objects of the form $j_{w'!*}(\mathrm{IC}_w)$ are isomorphic. Proposition \ref{prop:simples} tells us that $j_{w'!*}(\mathrm{IC}_w)$ and $j_{zw'!*}(\mathrm{IC}_w)$ are isomorphic for any $z \in P(w)$. Further, no other isomorphisms exist among the $\{j_{w'!*}(\mathrm{IC}_w)\}_{w, w' \in W}$, since for any fixed $w$, the same proposition tells us that the set of $z \in W$ for which $j_z^*(j_{w'!*}(\mathrm{IC}_w))$ is nonzero depends only on the image of $w'$ in $P(w)\backslash W$.
\end{proof}

\begin{example} 
When $\mathbf{G} = \mathbf{SL}_3$, Figure \ref{fig:sl3simples} is an explicit list of the simple objects in $\mathcal{A}_B$. In this case $w_0 = s_1s_2s_1 = s_2s_1s_2$.
\begin{figure}[ht]
\caption{The $19$ simple objects in $\mathcal{A}_B$ when $G = SL_3$.\label{fig:sl3simples}}
\begin{center}
 \begin{tabular}{||c || c | c | c | c | c | c ||} 
 \hline
  & $j_e^*(-)$ & $j_{s_1}^*(-)$ & $j_{s_2}^*(-)$ & $j_{s_1s_2}^*(-)$ & $j_{s_2s_1}^*(-)$ & $j_{w_0}^*(-)$ \\ [0.5ex] 
 \hline\hline
 $j_{e!*}(\mathrm{IC}_e)$ & $\mathrm{IC}_e$ & $\mathrm{IC}_e$ & $\mathrm{IC}_e$ & $\mathrm{IC}_e$ & $\mathrm{IC}_e$ & $\mathrm{IC}_e$ \\ 
 \hline
 $j_{e!*}(\mathrm{IC}_{s_1})$ & $\mathrm{IC}_{s_1}$ & $0$ & $\mathrm{IC}_{s_1}$ & $0$ & $0$ & $0$ \\ 
 \hline
 $j_{s_1!*}(\mathrm{IC}_{s_1})$ & $0$ & $\mathrm{IC}_{s_1}$ & $0$ & $0$ & $\mathrm{IC}_{s_1}$ & $0$ \\ 
 \hline
 $j_{s_1s_2!*}(\mathrm{IC}_{s_1})$ & $0$ & $0$ & $0$ & $\mathrm{IC}_{s_1}$ & $0$ & $\mathrm{IC}_{s_1}$ \\ 
 \hline
 $j_{e!*}(\mathrm{IC}_{s_2})$ & $\mathrm{IC}_{s_2}$ & $\mathrm{IC}_{s_2}$ & $0$ & $0$ & $0$ & $0$ \\ 
 \hline
 $j_{{s_2}!*}(\mathrm{IC}_{s_2})$ & $0$ & $0$ & $\mathrm{IC}_{s_2}$ &  $\mathrm{IC}_{s_2}$ & $0$ & $0$ \\ 
 \hline
 $j_{{s_2s_1}!*}(\mathrm{IC}_{s_2})$ & $0$ & $0$ & $0$ & $0$ & $\mathrm{IC}_{s_2}$ &  $\mathrm{IC}_{s_2}$\\ 
 \hline
 $j_{e!*}(\mathrm{IC}_{s_2s_1})$ & $\mathrm{IC}_{s_2s_1}$ & $0$ & $\mathrm{IC}_{s_2s_1}$ & $0$ & $0$ & $0$\\
 \hline
 $j_{s_1!*}(\mathrm{IC}_{s_2s_1})$ & $0$ & $\mathrm{IC}_{s_2s_1}$ & $0$ & $0$ & $\mathrm{IC}_{s_2s_1}$ & $0$ \\ 
 \hline
 $j_{s_1s_2!*}(\mathrm{IC}_{s_2s_1})$ & $0$ & $0$ & $0$ & $\mathrm{IC}_{s_2s_1}$ & $0$ & $\mathrm{IC}_{s_2s_1}$ \\ 
 \hline
 $j_{e!*}(\mathrm{IC}_{s_1s_2})$ & $\mathrm{IC}_{s_1s_2}$ & $\mathrm{IC}_{s_1s_2}$ & $0$ & $0$ & $0$ & $0$ \\ 
 \hline
 $j_{{s_2}!*}(\mathrm{IC}_{s_1s_2})$ & $0$ & $0$ & $\mathrm{IC}_{s_1s_2}$ &  $\mathrm{IC}_{s_1s_2}$ & $0$ & $0$ \\ 
 \hline
 $j_{{s_2s_1}!*}(\mathrm{IC}_{s_1s_2})$ & $0$ & $0$ & $0$ & $0$ & $\mathrm{IC}_{s_1s_2}$ &  $\mathrm{IC}_{s_1s_2}$\\ 
 \hline
 $j_{{e}!*}(\mathrm{IC}_{w_0})$ & $\mathrm{IC}_{w_0}$ & $0$ & $0$ & $0$ & $0$ &  $0$\\ 
 \hline
 $j_{{s_1}!*}(\mathrm{IC}_{w_0})$ & $0$ & $\mathrm{IC}_{w_0}$ & $0$ & $0$ & $0$ &  $0$\\ 
 \hline
 $j_{{s_2}!*}(\mathrm{IC}_{w_0})$ & $0$ & $0$ & $\mathrm{IC}_{w_0}$ & $0$ & $0$ &  $0$\\ 
 \hline
 $j_{{s_1s_2}!*}(\mathrm{IC}_{w_0})$ & $0$ & $0$ & $0$ & $\mathrm{IC}_{w_0}$ & $0$ &  $0$\\ 
 \hline
 $j_{{s_2s_1}!*}(\mathrm{IC}_{w_0})$ & $0$ & $0$ & $0$ & $0$ & $\mathrm{IC}_{w_0}$ &  $0$\\ 
 \hline
 $j_{{w_0}!*}(\mathrm{IC}_{w_0})$ & $0$ & $0$ & $0$ & $0$ & $0$ &  $\mathrm{IC}_{w_0}$\\ 
 \hline
\end{tabular}
\end{center}
\end{figure}
\end{example}

\subsection{Counting simple objects}
We now use Theorem \ref{thm:simpleobj} to give an explicit formula for the number of simple objects in $\mathcal{A}_{B}$.

\begin{corollary}\label{prop:counting}
The number of simple objects in $\mathcal{A}_B$ is
\begin{align}
    \sum_{w \in W} |P(w)\backslash W|.
\end{align}
\end{corollary}

\begin{remark}
In Type $A_n$, we can interpret the quantity in Corollary \ref{prop:counting} as the number of pairs of permutations in the symmetric group $S_{n+1}$ with no common rises in the sense of \cite{APcomb}. This is {A000275} in the OEIS:
$$1, 3, 19, 211, 3651, 90921, 3081513, 136407699, \dots$$
As a result, a generating function for the number of simple objects in $\mathcal{A}_B$ in this case is provided by the coefficients of a Bessel function, i.e. the reciprocal of $J_0(z)$ as in loc.\ cit.
\end{remark}

\section{Kazhdan-Laumon categories and Lusztig's module $M_{d}$}

The Grothendieck group $K_0(\kl)$ carries a natural $W$-action along with an action of the Hecke algebra $\fH$. In this section, we use the description of simple objects in $\kl$ to show that there exists an isomorphism $\eta : K_0(\kl)\otimes \mathbb{C} \to M_d^0$ of $\fH_q\otimes \mathbb{C}[W]$-modules.

\subsection{The modules $M^{+}_{d,q}$ and $\overline{M}_{d,q}^0$}

\begin{defn}
    Let $M_{d,q}^{+}$ be the $\mathcal{H}_q \otimes \mathbb{C}[W]$-submodule of $M_{d,q}$ generated by $A_{\tilde{w}}^\sharp$ for $\tilde{w} \in \tilde{W}$ with the property that for any $z \in W$, $\epsilon_z(\tilde{w})$ can be written as $w \cdot \lambda$ for some $w \in W, \lambda \geq 0$. It follows from \cite{L} that $M_{d,q}^0 \subset M_{d,q}^+$.

    Let $\overline{M}_{d,q}^0$ be the quotient of $M_{d,q}^+$ by the span of the elements $A_{\tilde{w}}^\sharp$ such that there is no pair $z, w \in W$ for which $\epsilon_z(w) = \tilde{w}$. Since it is easy to see that this is a $\mathcal{H}_q \otimes \mathbb{C}[W]$-submodule, we will continue to refer to the $\mathcal{H}_q$-action and the $W$-action by $\{\theta_y\}_{y \in W}$ on the quotient module $\overline{M}_{d,q}^0$.

    We use $\pi$ to refer to the quotient map from $M_{d,q}^+$ or $M_{d,q}^0$ to $\overline{M}_{d,q}^0$, and we write $\overline{A}^\sharp$ for $\pi(A^\sharp)$ whenever $A^\sharp \in M_{d,q}^+$.
\end{defn}

\subsection{The bijection\label{sec:thebijection}}
By the preceding section, recall that we can consider $K_0(\kl)$ as a $\mathcal{H}\otimes \mathbb{Z}[v, v^{-1}][W]$ module, with $v$ acting by the half-Tate twist $(-\tfrac{1}{2})$, $\mathcal{H}$ acting as in Proposition \ref{prop:heckeoaction}, and $W$ acting by the functors $\mathcal{F}_w$. We will then consider $K_0(\mathcal{A}_{\mathcal{P}}) \otimes_{\mathbb{Z}[v, v^{-1}]} \mathbb{C}$ where $v \mapsto q^{\frac{1}{2}}$.

\begin{theorem}
\label{thm:klophm}
There is a well-defined isomorphism $\eta : K_0(\mathcal{A}_{\mathcal{P}}) \otimes_{\mathbb{Z}[v, v^{-1}]} \mathbb{C} \to M_{d,q}^0$ of $\mathcal{H}_q\otimes \mathbb{C}[W]$-modules such that for any $w \in W$,
\begin{equation}
    \begin{split}
    \label{eqn:etastds}
        \eta([j_{e!}(\Delta_w)]) = A_w.
    \end{split}
\end{equation}

Further,
\begin{equation}
\label{eqn:etasimples}
    \begin{split}
        (\pi \circ \eta)([j_{e!*}(\mathrm{IC}_w)]) & = \overline{A}_w^\sharp
    \end{split}
\end{equation}
in $\overline{M}_{d,q}^0$ for any $w \in W$.
\end{theorem}

We note that (\ref{eqn:etasimples}) does not imply that $\eta$ sends $[j_{e!*}(\mathrm{IC}_w)]$ to $A_{w}^\sharp$; in fact, one can check that this is false for $\mathbf{G} = \mathbf{SL}_3$. Although we will show that $M_{d,q}^0$ and $\overline{M}_{d,q}^0$ are isomorphic, it is not always true that $A_{w}^\sharp$ lies in $M_{d,q}^0$. The intuitive role of $\overline{M}_{d,q}^0$ and $M_{d,q}^0$ will be clarified when these spaces are categorified in Section \ref{sec:anequiv}.

We illustrate in Figure \ref{fig:hexagons} the morphism $\pi \circ \eta : K_0(\mathcal{A}_\mathcal{P})\otimes \mathbb{C} \to \overline{M}_{d,q}^0$ in the case $\mathbf{G} = \mathbf{SL}_3$. The next subsection will be devoted to the proof of this result.
\begin{figure}[ht]
    \centering
    \begin{tikzpicture}[scale=1.3]
    \draw (0,0) -- (8,0);
    \draw (0,{sqrt(3)}) -- (8,{sqrt(3)});
    \draw (0,{2*sqrt(3)}) -- (8,{2*sqrt(3)});
    \draw (0,{3*sqrt(3)}) -- (8,{3*sqrt(3)});
    \draw (0,{4*sqrt(3)}) -- (8,{4*sqrt(3)});
    \draw (0,{5*sqrt(3)}) -- (8,{5*sqrt(3)});
    \draw (0,{4*sqrt(3)}) -- (1, {5*sqrt(3)});
    \draw (0,{2*sqrt(3)}) -- (3, {5*sqrt(3)});
    \draw (0,0) -- (5, {5*sqrt(3)});
    \draw (2,0) -- (7, {5*sqrt(3)});
    \draw (4,0) -- (8, {4*sqrt(3)});
    \draw (6,0) -- (8, {2*sqrt(3)});
    \draw (0, {2*sqrt(3)}) -- (2, 0);
    \draw (0, {4*sqrt(3)}) -- (4, 0);
    \draw (1, {5*sqrt(3)}) -- (6, 0);
    \draw (3, {5*sqrt(3)}) -- (8, 0);
    \draw (5, {5*sqrt(3)}) -- (8, {2*sqrt(3)});
    \draw (7, {5*sqrt(3)}) -- (8, {4*sqrt(3)});
    \draw [line width=0.5mm] (1, {sqrt(3)}) -- (3, {sqrt(3)});
    \draw [line width=0.5mm] (1, {3*sqrt(3)}) -- (3, {3*sqrt(3)});
    \draw [line width=0.5mm] (3, {3*sqrt(3)}) -- (4, {2*sqrt(3)});
    \draw [line width=0.5mm] (0, {2*sqrt(3)}) -- (1, {3*sqrt(3)});
    \draw [line width=0.5mm] (3, {sqrt(3)}) -- (4, {2*sqrt(3)});
    \draw [line width=0.5mm] (0, {2*sqrt(3)}) -- (1, {sqrt(3)});
    \node[align=center] at (3, {2*sqrt(3) + sqrt(3)/3}) {\huge {$e$}\\[0.16cm] \Huge \color{red}{$e$}};
    \node[align=center] at (2, {3*sqrt(3) - sqrt(3)/3}) {\huge {$e$}\\[0.16cm] \Huge \color{orange}{${s_1}$}};
    \node[align=center] at (1, {2*sqrt(3) + sqrt(3)/3}) {\huge {$e$}\\[0.16cm] \Huge \color{Dandelion}{${s_2s_1}$}};
    \node[align=center] at (1, {2*sqrt(3) - sqrt(3)/3}) {\huge {$e$}\\[0.16cm] \Huge \color{ForestGreen}{${w_0}$}};
    \node[align=center] at (2, {sqrt(3) + sqrt(3)/3}) {\huge {$e$}\\[0.16cm] \Huge \color{RoyalBlue}{$s_1s_2$}};
    \node[align=center] at (3, {2*sqrt(3) - sqrt(3)/3}) {\huge {$e$}\\[0.16cm] \Huge \color{Purple}{${s_2}$}};
    \node[align=center] at (1, {4*sqrt(3) - sqrt(3)/3}) {\huge {$s_2$}\\[0.16cm] \Huge \color{ForestGreen}{${w_0}$}};
    \node[align=center] at (2, {3*sqrt(3) + sqrt(3)/3}) {\huge {$s_2$}\\[0.16cm] \Huge \color{RoyalBlue}{${s_1s_2}$}};
    \node[align=center] at (3, {4*sqrt(3) - sqrt(3)/3}) {\huge {$s_2$}\\[0.16cm] \Huge \color{Purple}{${s_2}$}};
    \node[align=center] at (5, {4*sqrt(3) - sqrt(3)/3}) {\huge {$s_1s_2$}\\[0.16cm] \Huge \color{orange}{${s_1}$}};
    \node[align=center] at (4, {3*sqrt(3) + sqrt(3)/3}) {\huge {$s_1s_2$}\\[0.16cm] \Huge \color{Dandelion}{${s_2s_1}$}};
    \node[align=center] at (5, {2*sqrt(3) + sqrt(3)/3}) {\huge {$s_2s_1$}\\[0.16cm] \Huge \color{RoyalBlue}{${s_1s_2}$}};
    \node[align=center] at (6, {3*sqrt(3) - sqrt(3)/3}) {\huge {$s_2s_1$}\\[0.16cm] \Huge \color{Purple}{${s_2}$}};
    \node[align=center] at (5, {2*sqrt(3) - sqrt(3)/3}) {\huge {$s_1$}\\[0.16cm] \Huge \color{orange}{${s_1}$}};
    \node[align=center] at (4, {1*sqrt(3) + sqrt(3)/3}) {\huge {$s_1$}\\[0.16cm] \Huge \color{Dandelion}{${s_2s_1}$}};
    \node[align=center] at (4, {1*sqrt(3) - sqrt(3)/3}) {\huge {$s_1$}\\[0.16cm] \Huge \color{ForestGreen}{${w_0}$}};
    \node[align=center] at (7, {2*sqrt(3) - sqrt(3)/3}) {\huge {$s_2s_1$}\\[0.16cm] \Huge \color{ForestGreen}{${w_0}$}};
    \node[align=center] at (7, {4*sqrt(3) - sqrt(3)/3}) {\huge {$w_0$}\\[0.16cm] \Huge \color{ForestGreen}{${w_0}$}};
    \node[align=center] at (4, {5*sqrt(3) - sqrt(3)/3}) {\huge {$s_1s_2$}\\[0.16cm] \Huge \color{ForestGreen}{${w_0}$}};
    \end{tikzpicture}
    \caption{A picture illustrating the definition of $\eta'$ in the case $\mathbf{G} = \mathbf{SL}_3$. For every alcove $A \in \Xi$ in the picture which is labelled with the tuple $(v, w)$, we have $\eta'(j_{v!*}(\mathrm{IC}_w)) = \overline{A^\sharp}$. The colors indicate the orbits of the $A^\sharp$ under the operators $\theta_w$. The set $\Xi_{\mathrm{fin}}$ is in bold, and the fundamental alcove $A_e$ has a red label.}
    \label{fig:hexagons}
\end{figure}

\subsection{Proof of Theorem \ref{thm:klophm}}

We begin by defining $\eta' : K_0(\mathcal{A}_{\mathcal{P}}) \otimes \mathbb{C} \to \overline{M}_{d,q}^0$ for any $w, y \in W$ by
\begin{align}
    \eta'([j_{y!*}(\mathrm{IC}_{w})]) \mapsto \theta_{y^{-1}}\left(\overline{A}_{w}^\sharp\right).
\end{align}

To prove Theorem \ref{thm:klophm}, we will first prove that $\eta'$ is a well-defined and bijective morphism of $\mathbb{C}[W]$-modules onto $\overline{M}_{d,q}^0$. Then we will show that $\pi$, too, is an isomorphism of $\mathbb{C}[W]$-modules from $M_{d,q}^0$ to $\overline{M}_{d,q}^0$. Then we will show that $\pi \circ \eta = \eta'$, which gives that $\eta$ is well-defined and bijective. From this, we will deduce the fact that $\eta$ is a morphism of $\mathcal{H}_q$-modules.

\begin{lemma}\label{lem:etaprime}
    The map $\eta'$ is a well-defined isomorphism of $\mathbb{C}[W]$-modules.
\end{lemma}

\begin{proof}
We first claim that the stabilizer of any $[j_{e!*}(\mathrm{IC}_w)]$ under the action of the $\{\mathcal{F}_w\}_{w \in W}$ is equal to the stabilizer of $A_w^\sharp$ under the action of the $\{\theta_w\}_{w \in W}$. Note that Proposition \ref{prop:actionsharps} shows that the stabilizer of $A_w^\sharp$ under the latter action corresponds to the stabilizer of $A_w$ under the $*$-action of $W$ on $\Xi$. One can check that this stabilizer is $P(w)$, which is the stabilizer of $[j_{e!*}(\mathrm{IC}_w)]$ under the action of the $\{\mathcal{F}_w\}_{w \in W}$ by Theorem \ref{thm:simpleobj}.

Now note that this must also be equal to the stabilizer of $\overline{A}_w^\sharp$ under the action of the endomorphisms $\{\theta_{w}\}_{w \in W}$. Indeed, if $\theta_{w}(A_{y}^\sharp) = A_y^\sharp$ then the same is true for $\overline{A}_{y}^\sharp$. On the other hand, if $\theta_{w}(A_{y}^\sharp) \neq A_y^\sharp$, then $\theta_w(A_y^\sharp) = A_{\tilde{y}}^\sharp$ for some $\tilde{y} \in \tilde{W}$ not equal to $y$. But it is clear that the distinct canonical basis elements $A_{y}^\sharp$ and $A_{\tilde{y}}^\sharp$ in $\overline{M}_{d,q}^+$ do not become identified in $\overline{M}_{d,q}^0$, by the definition of $\overline{M}_{d,q}^0$, so we cannot then have $\theta_w(\overline{A}_{y}^\sharp) = \overline{A}_{y}^\sharp$.

Finally, we note that the map is surjective, as by Definition \ref{def:epsilonindices}, the image of $\eta'$ contains all elements of the form $\overline{A}_{\tilde{w}}^\sharp$ for those $\tilde{w} \in \tilde{W}$ such that $A_{\tilde{w}}^\sharp$ was not annihilated under the quotient map $M_{d,q}^+ \to \overline{M}_{d,q}^0$. 
\end{proof}

\begin{lemma}\label{lem:pi}
    The morphism $\pi : M_{d,q}^0 \to \overline{M}_{d,q}^0$ is an isomorphism of $\mathbb{C}[W]$-modules.
\end{lemma}

\begin{proof}
    Note first that $\pi$ is clearly a surjective morphism of $\mathbb{C}[W]$-modules. To show it is an isomorphism, it is sufficient to check that $\dim M_{d,q}^0 \leq \dim \overline{M}_{d,q}^0$. By Lemma \ref{lem:etaprime} and Corollary \ref{prop:counting},
    \begin{align}
        \dim \overline{M}_{d,q}^0 = \sum_{w \in W} |P(w)\backslash W|.
    \end{align}
    We now claim that $M_{d,q}^0$ is spanned by elements of the form $\theta_{y^{-1}}(A_w)$ where $y$ is a the minimal-length representative of some coset $P(w)y$ of $P(w)$ in $W$. Since there are $\sum_{w \in W} |P(w) \backslash W|$ such elements, this will prove the inequality above.

    We do this by induction on the Bruhat order: we claim that for any $w \in W$, $\mathrm{span}\{\theta_{y^{-1}}(A_{w'})\}_{y \in W, w' \geq w}$ is spanned by the subset of these elements with minimal-length $y$ in the coset $P(w)y$ as above. For $w = w_0$, this is trivial as $P(w) = \varnothing$.

    Suppose now that it is true for some $w$. Then by further induction on the length of $y$ it is enough to show that for any $s \in P(w)$, $\theta_s(A_w)$ is a linear combination of $A_{ws}$, $A_w$, and $\theta_s(A_{ws})$. Indeed, by the definition of $\theta_s$ in (\ref{eqn:deftheta}), one can check that, since $\ell(ws) > \ell(w)$,
    \begin{align}
        \theta_s(A_w) & = A_w + v^{-1}A_{ws} - v^{-1}\theta_s(A_s),
    \end{align}
    completing the proof.
\end{proof}

Now we define $\eta = \pi^{-1} \circ \eta'$. By Lemmas \ref{lem:etaprime} and \ref{lem:pi}, we know $\eta$ is an isomorphism of $\mathbb{C}[W]$-modules. The following proposition will follow from Proposition \ref{prop:je}, which is proved in the next subsection, and so we will delay its proof until then.

\begin{proposition}
\label{prop:etastds}
For any $w \in W$,
\begin{align*}
    \eta([j_{e!}(\Delta_w)]) & = A_w
\end{align*}
\end{proposition}

Before we use Proposition \ref{prop:etastds} to show that $\eta$ is a morphism of $\mathcal{H}_q$-modules to prove Lemma \ref{lem:etahmodules} below, we need the following straightforward consequence of Polishchuk's work on ``good representations" in \cite{P}. 

\begin{lemma}[\cite{P}, Theorem 11.5.1]\label{lem:goodrep}
The elements $\{[j_{z!}(\Delta_w)]\}_{z, w \in W}$ generate $K_0(\mathcal{A}_\mathcal{P}) \otimes_{\mathbb{Z}[v, v^{-1}]} \mathbb{C}$ whenever $v$ is specialized to a value which is not a root of unity (in particular, in our situation where $v \mapsto q^{\frac{1}{2}}$). 
\end{lemma}
\begin{proof}
    In Theorem 11.5.1 of \cite{P}, it is shown that $K_0(\mathcal{A}_\mathcal{P})$ must be generated by objects of the form $[j_{z!}(A)]$ for $A \in \mathcal{P}$ so long as the action defined by the $F_{w,!}$ functors factors through the Hecke algebra $\mathcal{H}_q$. Although this is not true for the action of these functors on $K_0(\mathcal{A})$, it holds for $K_0(\mathcal{A}_\mathcal{P})$ by Proposition \ref{prop:f!conv}. Since $\mathcal{P}$ is generated by standard objects $\{\Delta_w\}_{w \in W}$ and their Tate twists, this proves the lemma.
\end{proof}
Note that it is not true that the elements in Lemma \ref{lem:goodrep} generate $K_0(\mathcal{A}_\mathcal{P})$ alone before specializing at $q^{\frac{1}{2}}$; this follows from the example provided in the appendix to \cite{P}. The question of whether the Grothendieck group $K_0(\mathcal{A})$ of the full Kazhdan-Laumon category is generated by a similar collection is the subject of much of the work in \cite{P}, and this has now been addressed further in \cite{CMF3}.

\begin{lemma}
\label{lem:etahmodules}
$\eta$ is a morphism of $\mathcal{H}_q$-modules.
\end{lemma}

\begin{proof}
By Section \ref{sec:cato},
\begin{equation*}
\begin{split}
    \eta(\tilde{T}_s\cdot [j_{e!}(\Delta_w)]) & = \eta([j_{e!}(\nabla_s * \Delta_w)])\\
    & = \begin{cases}
    \eta([j_{e!}(\Delta_{sw})]) & \ell(sw) < \ell(w),\\
    \eta([j_{e!}(\Delta_{sw})] + [j_{e!}(\Delta_w(-\tfrac{1}{2})] - [j_{e!}(\Delta_w(\tfrac{1}{2})]) & \ell(sw) > \ell(w)
    \end{cases}\\
    & = \begin{cases}
    A_{sw} & \ell(sw) < \ell(w),\\
    A_{sw} + (q^{1/2} - q^{-1/2})A_w & \ell(sw) > \ell(w).
    \end{cases}
\end{split}
\end{equation*}
Comparing this with (\ref{eqn:heckeaction}), we conclude
\begin{equation}\label{eqn:heckeactiononbasics}
    \eta(\tilde{T}_s\cdot [j_{e!}(\Delta_w)]) = \tilde{T}_s \cdot \eta([j_{e!}(\Delta_w)])
\end{equation}
for any $w \in W$, $s \in S$. 

By Lemma \ref{lem:goodrep}, the elements $[j_{z!}(\Delta_w)]$ for $z, w \in W$ span the space $K_0(\mathcal{A}_{\mathcal{P}})\otimes \mathbb{C}$. Since the $W$-action commutes with the action of $\mathcal{H}_q$, the equation (\ref{eqn:heckeactiononbasics}) implies that $\eta$ is a $\mathcal{H}_q$-module homomorphism on all of $K_0(\mathcal{A}_{\mathcal{P}}) \otimes \mathbb{C}$, as desired. 
\end{proof}

Combining Lemmas \ref{lem:etaprime} and \ref{lem:etahmodules} (which, to recall, show together that $\eta$ is a $\mathbb{C}[W]$-module isomorphism) with Proposition \ref{prop:etastds}, Theorem \ref{thm:klophm} is proved. It remains to prove Proposition \ref{prop:etastds}; this will be our main focus in the next subsection.

\subsection{Restriction to fundamental alcoves} Let $\Xi_{\mathrm{fin}}$ denote the set of alcoves $\{A_w\}_{w \in W}$ which we call the fundamental alcoves. In this subsection, we define a map $J_e : M_{d,q}^0 \to K_0(\mathcal{P})\otimes_{\mathbb{Z}[v, v^{-1}]} \mathbb{C}$ (which can be interpreted as a sort of restriction to $\Xi_\mathrm{fin}$) such that the diagram
\begin{align}\label{jediagram}
   \begin{tikzcd}[ampersand replacement=\&]
             K_0(\kl)\otimes_{\mathbb{Z}[v, v^{-1}]} \mathbb{C} \arrow[r, "\eta"] \arrow[dr, "j_e^*"'] \& M_{d,q}^0 \arrow[d, "J_e"]\\
             \& K_0(\mathcal{P})\otimes_{\mathbb{Z}[v, v^{-1}]} \mathbb{C}
    \end{tikzcd}
    \end{align}
commutes. This gives an interpretation of the functor $j_e^*$ on the $M_d^0$ side of the bijection from Theorem \ref{thm:klophm}, and will allow us to prove Proposition \ref{prop:etastds}.

\begin{defn}
We define the map $J_e : M_{d,q}^0 \to K_0(\mathcal{P}) \otimes_{\mathbb{Z}[v, v^{-1}]} \mathbb{C}$ as the unique $\mathbb{C}$-linear map satisfying
\begin{align*}
    J_e(A_w) & = \begin{cases}
    [\Delta_w] & \text{if $w \in W$}\\
    0 & \text{if $w \in \tilde{W} \setminus W$},
    \end{cases}
\end{align*}
extending linearly to the domain $M_{d,q}^0$.

Let $J_e' : K_0(\mathcal{P})\otimes_{\mathbb{Z}[v, v^{-1}]} \mathbb{C} \to M_{d,q}^0$ be the unique section of $J_e$ whose image lies in the subspace of $M_d^0$ spanned by $\{A_w\}_{w \in W}.$ Let $\xi = J_e' \circ J_e$ be the projection onto the subspace of $M_{d,q}^0$ spanned by $\{A_w\}_{w \in W}$. Finally, let $\rho = \mathrm{id} - \xi$ be the projection onto the subspace of $M_{d,q}^0$ spanned by $\{A_{\tilde{w}}\}_{\tilde{w} \in \tilde{W}\setminus W}$.
\end{defn}

\begin{proposition}\label{prop:je}
The diagram (\ref{jediagram}) commutes, i.e. $J_e \circ \eta = j_e^*$.
\end{proposition}

\begin{proof}
By the linearity of each of these maps, it is enough to check this statement on the elements $\{[j_{z!*}(\mathrm{IC}_w)]\}_{z, w \in W}$, which span $K_0(\mathcal{A}_\mathcal{P})$ under $\mathbb{Z}[v, v^{-1}]$ by Theorem \ref{thm:simpleobj}. By the same result, we know that
\begin{align*}
    j_{e}^*([j_{z!*}(\mathrm{IC}_w)]) & = \begin{cases}
    [\mathrm{IC}_w] & z \in P(w)\\
    0 & z\not\in P(w),
    \end{cases}
\end{align*}
while $J_e(\eta([j_{z!*}(\mathrm{IC}_w)])) = J_e(\theta_{z^{-1}}(A_w^\sharp))$. Since $\pi \circ \eta = \eta'$ and it is easy to check that $J_e$ is zero on any objects in the kernel of $\pi$ by the definition of $\overline{M}_{d,q}^0$, it is enough to show that
\begin{align*}
    J_e(\theta_{z^{-1}}(A_w^\sharp)) & = \begin{cases}
    [\mathrm{IC}_w] & z \in P(w)\\
    0 & z \not\in P(w).
    \end{cases}
\end{align*}
In the case where $z \in P(w)$, we know by the proof of Lemma \ref{lem:etaprime} that $\theta_z(A_w^\sharp) = A_w^\sharp$, so this reduces to showing $J_e(A_w^\sharp) = [\mathrm{IC}_w]$. By the Kazhdan-Lusztig conjectures, formula (1.5.b) in \cite{KL1} holds, and so our claim follows from the remark following 13.10 in \cite{L1} combined with 11.19 of loc.\ cit. (for $\epsilon = 0$), c.f.\ 11.15 of \cite{L}.

On the other hand, if $z \not\in P(w)$, then $\theta_{z^{-1}}(A_w^\sharp) = (z^{-1}*A_w)^\sharp$. It is easy to see that if $a * A_w \neq A_w$ for $a, w \in W$, then $a * A_w$ is not dominated by any alcove in $\{A_{w'}\}_{w' \in W}$ under the partial ordering of alcoves given in \cite{L1}. As a result, we cannot have any of the alcoves in $\{A_{w'}\}_{w' \in W}$ occurring with nonzero coefficient in $(z * A_w)^\sharp$, and so $J_e(\theta_{z^{-1}}(A_w^\sharp)) = 0$ in this case, as desired.
\end{proof}

\begin{corollary}\label{cor:zero}
An element $C \in M_d^0$ is zero if and only if $J_e(\theta_z(C)) = 0$ for all $z \in W$.
\end{corollary}
\begin{proof}
By the previous proposition and the fact that $\eta$ is an isomorphism of $\mathbb{C}[W]$-modules, this reduces to the fact that $[\mathcal{F}] = 0$ for $\mathcal{F} \in \kl$ if and only if $j_z^*(\mathcal{F}) = 0$ for all $z \in W$.
\end{proof}

Now to prove Proposition \ref{prop:etastds}, we will need the following; it is an easy computation from the definition of $\theta_s$.

\begin{lemma}\label{lem:jebasecase}
For any $w \in W$, $s \in S$,
\begin{align*}
    J_e(\theta_s(A_w)) & = [\Delta_w * \nabla_s(\tfrac{1}{2})].
\end{align*}
\end{lemma}
We now wish to upgrade this result by induction and show that $J_e(\theta_z(A_w)) = [\Delta_w * \nabla_z(\tfrac{\ell(z)}{2})]$ for any $w, z \in W$. To do so, we will need to keep track of possible positions of alcoves after applying $\theta_z$ to an alcove in $\Xi_{\mathrm{fin}}$. To make this precise, we introduce the following definition and subsequent lemmas.

\begin{defn}
For any $v \in W$ and reduced word $\overline{v}$ for $v$, recall that we can define a subset of positive roots $R(\overline{v})$. If $\overline{v} = s_{i_1}\dots s_{i_k}$, we set
\begin{align*}
    R(\overline{v}) & = \{\alpha_{i_1}, s_{i_1}(\alpha_{i_2}), \dots, s_{i_1} \dots s_{i_{k-1}}(\alpha_{i_k})\}.
\end{align*}
Note that if $\overline{v}$ is reduced and $s$ is a simple reflection with $\ell(sv) > \ell(v)$, then
\begin{align*}
    R(s\overline{v}) & = sR(\overline{v}) \cup \{\alpha_s\}.
\end{align*}
We will sometimes use $R(v)$ to denote $R(\overline{v})$ for some choice of reduced word $\overline{v}$ for $v \in W$. 
\end{defn}
\begin{defn}
Given any subset $R$ of positive roots, we define $\Xi^+(R) \subset \Xi$ to be the set of alcoves $A$ satisfying the following two conditions. First, we require that $A \not\in \Xi_{\mathrm{fin}}$. Second, we require that $A$ lies in the same Weyl chamber as some element of $R$. By this definition, each each $\Xi^+(R)$ will be a union of Weyl chambers but with $\Xi_{\mathrm{fin}}$ excised.
\end{defn}

One can easily check for any two reduced expressions $\overline{w}$ and $\overline{w}'$ for $w \in W$ that $\Xi^+(R(\overline{w})) = \Xi^+(R(\overline{w}'))$.
\begin{defn}
For any $w \in W$, let $\Xi_w^+ = \Xi^+(R(\overline{w}))$ for a choice of reduced expression $\overline{w}$ for $w$.
\end{defn}

The following is a straightforward computation which follows from the definition of $\theta_s$.
\begin{lemma}\label{lem:thetasxiplus}
If $v \in W$ and $s \in S$ with $\ell(sv) > \ell(v)$, then
\begin{align*}
    \theta_s(\mathrm{span}(\Xi^+_v) \cup \Xi_{\mathrm{fin}})) & \subset \mathrm{span}(\Xi^+_{sv} \cup \Xi_{\mathrm{fin}}),\\
    \xi (\theta_s(\mathrm{span}(\Xi^+_v))) & = 0.
\end{align*}
\end{lemma}

\begin{corollary}\label{cor:thetasr}
For any $w, z \in W$,
\begin{align*}
    \rho(\theta_z(A_w)) \in \mathrm{span}(\Xi^+_z).
\end{align*}
\end{corollary}
\begin{proof}
We go by induction on $\ell(z)$. For $\ell(z) = 0$, there is nothing to prove. Now suppose the result holds for $z$, and let $s \in S$ be such that $\ell(sz) > \ell(z)$. Then since $\rho + \xi = \mathrm{id}$,
\begin{align*}
\rho(\theta_{sz}(A_w)) & = (\rho \circ \theta_s)[\rho(\theta_z(A_w)) + \xi(\theta_z(A_w))].
\end{align*}
Clearly $\xi(\theta_z(A_w)) \in \mathrm{span}(\Xi_{\mathrm{fin}})$, and by our induction hypothesis, $\rho(\theta_z(A_w)) \in \mathrm{span}(\Xi^+_z)$. So Lemma \ref{lem:thetasxiplus} guarantees that $\theta_s[\rho(\theta_z(A_w)) + \xi(\theta_z(A_w))]$ lies in the span of $\Xi^+_{sz}\cup \Xi_{\mathrm{fin}}$, and therefore its image under $\rho$ lies in $\mathrm{span}(\Xi^+_{sz})$.
\end{proof}

With these new definitions and lemmas in hand, we can now provide an inductive proof of Proposition \ref{prop:etastds}.

\begin{proof}[Proof of Proposition \ref{prop:etastds}]
By the definition of $j_{e!}$,
\begin{align*}
    j_e^*(\mathcal{F}_z(j_{e!}(\Delta_w)]))) & = \Delta_w * \nabla_{z^{-1}}(\tfrac{\ell(z)}{2}),
\end{align*}
so by Corollary \ref{cor:zero}, to show that $\eta([j_{e!}(\Delta_w)]) = A_w$ it is enough to show that 
\begin{align}\label{eqn:jeeqn}
    J_e(\theta_z(A_w)) & = [\Delta_w * \nabla_{z^{-1}}(\tfrac{\ell(z)}{2})],
\end{align}
for all $w, z \in W$. We will fix $w$ and proceed by induction on $\ell(z)$. For $\ell(z) = 0$, this is immediate from the definition of $J_e$.

Suppose then that we know (\ref{eqn:jeeqn}) holds for some $z$. Suppose $s \in S$ is such that $\ell(sz) > \ell(z)$. Then by this assumption along with Lemma \ref{lem:jebasecase} (extended by linearity, using linearity of the convolution $- * \nabla(\tfrac{1}{2})$) that
\begin{align*}
(J_e \circ \theta_s \circ \xi \circ \theta_z)(A_w) & = J_e(\theta_s(J_e'([\Delta_w * \nabla_{z^{-1}}(\tfrac{\ell(z)}{2})])))\\
& = [\Delta_w * \nabla_{z^{-1}}(\tfrac{\ell(z)}{2}) * \nabla_s(\tfrac{1}{2})]\\
& = [\Delta_w * \nabla_{(sz)^{-1}}(\tfrac{\ell(sz)}{2})].
\end{align*}
To complete the proof, it remains to show that
\begin{align*}
    J_e(\theta_{sz}(A_w)) & = (J_e \circ \theta_s \circ \xi \circ \theta_z)(A_w).
\end{align*}
By Corollary \ref{cor:thetasr}, we know $\rho(\theta_z(A_w)) \in \mathrm{span}(\Xi^+_z)$. By the first part of Lemma \ref{lem:thetasxiplus}, we then have
\begin{align*}
    J_e(\theta_{sz}(A_w)) - J_e(\theta_s(\xi(\theta_z(A_w)))) & = (J_e\circ \theta_s)[\theta_z(A_w) - \xi(\theta_z(A_w))]\\
    & = (J_e\circ \theta_s)(\rho(\theta_z(A_w)))\\
    & \in J_e(\theta_s(\mathrm{span}(\Xi^+_z))),
\end{align*}
which is zero by the second part of Lemma \ref{lem:thetasxiplus}.
\end{proof}

\section{The Schwartz space of the basic affine space\label{sec:schwartz}}

In this section, we recall the definition of the Schwartz space of the basic affine space $\mathcal{S}$ as defined in \cite{BK}. We then focus on Iwahori-invariants $\mathcal{S}^{I \times \mathbf{T}(\mathcal{O})}$ of this space, and revisit a claim made in loc.\ cit. about its connection to $M_d$. This will then serve as a bridge to a more direct connection, which we will consider in the final subsection, between $\mathcal{S}$ and the Kazhdan-Laumon category $\mathcal{A}$.

\subsection{Preliminaries from \cite{BK}}
In this section, let $k$ be a non-Archimedean local field with residue field $\kappa = \mathbb{F}_q$, and $\mathcal{O}$ its ring of integers. Let $\pi$ be a uniformizer of $k$, and suppose the norm on $k$ is chosen so that $||\pi|| = q^{-1}$. 

In this section only, we write $G_k$ for the base change of $\mathbf{G}$ to $k$; we will consider the basic affine space $\X = (\mathbf{G}/\mathbf{U})(k)$. Let $I \subset \mathbf{G}(k)$ be an Iwahori subgroup. The setup of \cite{BK} begins with considering the space $\mathcal{S}_c$ of locally constant compactly supported functions on $\X$; let $\mathcal{S} \supset \mathcal{S}_c$ be the Schwartz space of the basic affine space defined in loc.\ cit.

\begin{defn}
For any $n \in \mathbb{Z}$, let $\psi_n$ be a choice of additive character of conductor $n$ of the field $k$ (i.e. a character for which $\pi^n \mathcal{O} \subset \ker \psi_n$ but $\pi^{n-1}\mathcal{O} \not\subset \ker \psi_n$).
\end{defn}

In \cite{BK}, a $W$-action on $\mathcal{S}$ is defined via Fourier transforms $\Phi_\alpha$ associated to any $s_\alpha \in S$ as follows. Let $P_{\alpha}(k)$ be the parabolic subgroup of $\mathbf{G}(k)$ naturally corresponding to $s_\alpha$, and let $Q_{\alpha}(k) = [P_\alpha(k), P_\alpha(k)]$. Then $\pi_\alpha : \X \to \mathbf{G}(k)/Q_\alpha(k)$ is a fibration with fibers isomorphic to $\mathbb{A}^2_k \setminus \{(0, 0)\}$; the Fourier transform $\Phi_\alpha$ is defined fiberwise. Because the definition in \cite{BK} (c.f.\ the more explicit definition in \cite{KForms}, which we will now slightly modify for our choice of the character $\psi_1$) depends on a choice of character and normalization, in the following definition we choose conventions explicitly for the purposes of the present paper.
\begin{defn}\label{def:fouriertransforms}
Define the Fourier transform operator $\Phi_\alpha : \mathcal{S} \to \mathcal{S}$ associated to a simple root $s_\alpha \in S$ by
\begin{align}
    (\Phi_{\alpha}(f))(x) & = q\int_{\pi_\alpha^{-1}(\pi_\alpha(x))} \psi_1(\langle x, x'\rangle_{s_\alpha})f(x')dx',
\end{align}
where we use the identification of the fiber $\pi_\alpha^{-1}(\pi_\alpha(x))$ with $V_s(x) - \{0\}$, where $V_s(x)$ is a two-dimensional vector space with a volume form defined by the pairing $\langle, \rangle_{s_\alpha} : V \times V \to k$ (following the setup and notation of Section 1.1.2 of \cite{KForms}).

We define $\Phi_w$ for any $w \in W$ by $\Phi_{w} = \Phi_{\alpha_1} \cdots \Phi_{\alpha_k}$ where $w = s_{\alpha_1} \dots s_{\alpha_k}$ is a reduced expression.
\end{defn}

\begin{remark}
The definition of the Fourier transform endomorphism $\Phi_\alpha$ in \cite{BK} depends on a choice of additive character of $k$. From now on, we will always use the character $\psi_1$ and the normalization (by $q$) chosen in Definition \ref{def:fouriertransforms}. This is the only choice of character for which Theorem \ref{thm:bksphm} will hold as it is written. The usual choice of character in the literature (as is used in Example 3.2 of \cite{BK} and throughout \cite{KForms} and \cite{D}, for example) is $\psi_0$. However, this choice is not compatible with the normalizations chosen in the statement of Theorem \ref{thm:bksphm}, as we explain in Example \ref{eg:doesntwork}. The reader should beware that our nonstandard choice of $\psi$ is such that, in the $\mathbf{SL}_2$ case,
\begin{align*}
    \Phi_\alpha(\chi_{\mathcal{O}\times \pi \mathcal{O}}) & = \chi_{\mathcal{O}\times \pi\mathcal{O}},\\
    \Phi_\alpha(\chi_{\mathcal{O}\times \mathcal{O}}) & = q\chi_{\pi\mathcal{O}\times \pi\mathcal{O}}\neq \chi_{\mathcal{O}\times \mathcal{O}},
\end{align*}
which does not agree with the computation in Example 3.2 of \cite{BK} due to the discrepancy between the characters $\psi_1$ and $\psi_0$, as well as the constant $q$ by which we multiply the integral in Definition \ref{def:fouriertransforms}. 
\end{remark}

The Iwahori subgroup $I$ acts naturally on $\X$ on the left, while $\mathbf{T}(\mathcal{O})$ acts on the right. We consider now the $I$-invariant subspace $\mathcal{S}^{I \times \mathbf{T}(\mathcal{O})}$, as is studied in Section 4 of \cite{BK}. Recall, as explained in loc.\ cit., that the $I$-orbits on $\X$ are indexed by $\aW$; for any $w \in \aW$ we can associate the $I$-orbit $Iw\mathbf{U}(k) \subset \mathbf{X}(k)$. We also know that for all $w \in \aW$ and $f \in \bks$, we can define the convolution $\chi_{IwI} * f$, where $\chi_{IwI}$ is the indicator function on the $I$-orbit $IwI \subset G(k)/I$. This defines an action of the affine Hecke algebra $\aH_q$, with $T_w$ acting such that
\begin{align}
\label{eqn:convolutionaction}
    T_{w^{-1}}^{-1}(f) & = (-q^{-1/2})^{\ell(w)}\chi_w * f.
\end{align}

\subsection{Connection to $M_d$}

In Lemma 4.3 and Corollary 4.6 of \cite{BK}, it was claimed that $\uphm{q} \cong \mathcal{S}(X)^{I\times \mathbf{T}(\mathcal{O})}$ as $\mathcal{H}_q \otimes \mathbb{C}[W]$-modules, where on the Schwartz space side the $\mathcal{H}_q$-action is given by convolution and the $W$-action on the  given by the $\Phi_w$, while on the $M_{d,q}$ side the $W$-action is given by the $\theta_w$. We reiterate that this result is not true if the ``usual" definition of the Fourier transforms using the character $\psi_0$ is used, as we explain in the following example. 

\begin{example}\label{eg:doesntwork}
Suppose $\mathbf{G} = \mathbf{SL}_2$, so $\X \cong \mathbb{A}_k^2\setminus\{(0, 0)\}$. For each $w \in \widetilde{A}_1 = \langle s_1, s_0\rangle$, there is an Iwahori-orbit $Iw\mathbf{U}(k) \subset \X$. These sets are either of the form $\pi^n\mathcal{O} \times \pi^n\mathcal{O} \setminus \pi^n\mathcal{O} \times \pi^{n+1}\mathcal{O}$ or $\pi^n \mathcal{O} \times \pi^{n+1}\mathcal{O} \setminus \pi^{n+1}\mathcal{O}\times \pi^{n+1}\mathcal{O}$ for some $n \in \mathbb{Z}$.

Lemma 4.3 in \cite{BK} asserts that the map $\gamma : \mathcal{S}(X)^{I\times \mathbf{T}(\mathcal{O})} \to \uphm{q}$ takes some multiple of the indicator function on each Iwahori orbit $Iw\mathbf{U}(k)$ to some alcove $A \in \Xi$. Accordingly, set $A^* = \gamma(-q^{-\frac{1}{2}}\chi_{Is_1\mathbf{U}(k)})$. (Note that $Is_1\mathbf{U}(k) = \mathcal{O}\times \mathcal{O} \setminus \mathcal{O}\times \pi\mathcal{O}$.)

Following the computations in Section \ref{sec:sl2phm}, one can compute that we must have $\gamma(-q^{-\frac{1}{2}}\chi_{\mathcal{O}\times \mathcal{O}}) = (A^*)^\sharp$. Since $\chi_{\mathcal{O}\times \mathcal{O}}$ is $G(\mathcal{O})$-invariant, we know that $T_{s_1}(\chi_{\mathcal{O}\times \mathcal{O}}) = -q^{-1/2}\chi_{\mathcal{O}\times \mathcal{O}}$. So if $\gamma$ is to be an isomorphism of $\mathcal{H}_q$-modules, we must have that $T_{s_1}((A^*)^\sharp) = -q^{-1/2}(A^*)^\sharp$, meaning $A^* = A_n$ for some odd $n$, by the computations in Section \ref{sec:sl2phm}.

This is impossible, however, since if the character $\psi_0$ is used, then $\Phi_{\alpha}(\chi_{\mathcal{O} \times \mathcal{O}}) = \chi_{\mathcal{O}\times \mathcal{O}}$, and so we expect $\theta_{s_1}(A_n^\sharp) = A_n^\sharp$. But this only holds for $n = 0$, which is not odd, contradicting the existence of such a $\gamma$. This example is intended to justify Definition \ref{def:fouriertransforms}, and specifically our nonstandard choice of additive character, showing that such a choice is necessary if Corollary 4.6 in \cite{BK} is to hold.
\end{example}

We can fix this by our use of the character $\psi_1$ in Definition \ref{def:fouriertransforms}, thereby yielding the following result. For ease of notation, we denote by $\delta_w$ the function $(-q^{1/2})^{d(A_e, A_w)}\chi_{Iw\mathbf{U}(k)} \in \bks$.
\begin{theorem}[\cite{BK}]
\label{thm:bksphm}
The map $\Psi : M_c \to \mathcal{S}_c^{I\times \mathbf{T}(\mathcal{O})}$ defined on alcoves $A_w$ by
\begin{align*}
    \Psi(A_w) & = \delta_w,
\end{align*}
extends to an isomorphism between $M_{d,q}$ and $\bks$ of $\mathcal{H}_q \otimes \mathbb{C}[\Gamma \rtimes W]$-modules.
\end{theorem}

\begin{example}
We now explain this bijection explicitly in the case of $\mathbf{SL}_2$; its proof in general follows from this computation. In this case, $W = \langle s_1\rangle$; we also denote by $s_0$ the simple reflection associated to the affine root. Let $A_n$ be the alcove $(n, n + 1)$. Note that for any $n$, $A_{2n} = A_{(s_1s_0)^n}$, while $A_{2n-1} = A_{s_1(s_1s_0)^n}$.

The action of $\mathcal{H}_q$ on $M_{d,q}$ is given by the formulas of Section \ref{sec:examplesl2} specialized at $v = q^{\frac{1}{2}}$. Comparing these with the action of $T_{s_1}$ on $\bks$ whose inverse is given by convolution with $-q^{-1/2}\chi_{Is_1I}$, one can check that the actions agree; the same goes for the action of $\Gamma$.

It remains to check that the $W$-actions agree, which is only true when using the character $\psi_1$ in the definition of $\Phi_{s_1}$. By applying $\Psi$ to the formula (\ref{eqn:sl2simple}),
\begin{align*}
\Psi(A_n^\sharp) & = \begin{cases}
(-q^{1/2})^n\chi_{(\pi^{n/2}\mathcal{O})\times (\pi^{n/2 + 1}\mathcal{O})} & \text{if $n$ is even,}\\
(-q^{1/2})^n\chi_{(\pi^{(n+1)/2}\mathcal{O})\times (\pi^{(n+1)/2}\mathcal{O})}& \text{if $n$ is odd.}
\end{cases}
\end{align*}
Since for $(x, y) \in k\times k$,
\begin{equation*}
    \begin{split}
        \Phi_{s_1}(\chi_{\pi^k \mathcal{O}\times \pi^{k+1}\mathcal{O}})(x, y) & = q\int_{a \in \pi^k\mathcal{O}}\int_{b \in \pi^{k+1}\mathcal{O}} \psi_1(xb - ya)\\
        & = q\left(\int_{a \in \pi^k\mathcal{O}}\psi_1(-ya)\right)\left(\int_{b \in \pi^{k+1}\mathcal{O}}\psi_1(xb)\right)\\
        & = q\cdot q^{-2k - 1}\chi_{\pi^{-k}\mathcal{O}}(x)\chi_{\pi^{-k+1}\mathcal{O}}(y)\\
        & = q^{-2k}\chi_{\pi^{-k}\mathcal{O}\times \pi^{-k+1}\mathcal{O}}(x, y),
    \end{split}
\end{equation*}
which implies that
\begin{align*}
\Phi_{s_1}\Psi(A_n^\sharp) & = \Psi(A_{-n}^\sharp) = \Psi(\theta_{s_1}(A_{-n}^\sharp))
\end{align*}
for $n$ even, while a similar proof works for $n$ odd.
\end{example}

\subsection{Schwartz space and Kazhdan-Laumon categories\label{sec:eisenstein}}

We now translate Theorem \ref{thm:klophm} into a statement about the Schwartz space of the basic affine space, since we can relate $M_{d,q}$ to $\bks$ by Theorem \ref{thm:bksphm}.

\begin{theorem}
\label{thm:kloschwartz}
There is an injection $\Theta : K_0(\kl)\otimes_{\mathbb{Z}[v, v^{-1}]} \mathbb{C} \to \bks$ of $\fH_q\otimes \mathbb{C}[W]$-modules given by
\begin{equation}
\label{eqn:deltadelta}
    \begin{split}
        \Theta([j_{e!}(\Delta_w)]) & = \delta_w
    \end{split}
\end{equation}
for all $w \in W$.
\end{theorem}
\begin{proof}
This follows directly from Theorems \ref{thm:bksphm} and \ref{thm:klophm}. The last part follows in particular by comparing equation (\ref{eqn:etastds}) with Theorem \ref{thm:bksphm} and the definition of $\delta_w$.
\end{proof}

Now we place Theorem \ref{thm:klophm} in this context, rephrasing the result in terms of a more direct connection between the Kazhdan-Laumon category $\mathcal{A}$ containing $\mathcal{A}_B$, and the Schwartz space $\mathcal{S} \supset \bks$.

Grothendieck's sheaf-function correspondence tells us that for any $\mathcal{F} \in \mathcal{P}$, we can produce a function $\text{tr}(\mathcal{F}) : (\mathbf{G}/\mathbf{U})(\mathbb{F}_q) \to \overline{\mathbb{Q}}_\ell$ (we implicitly pull back along $\pi : G/U \to G/B$ throughout this section to work only with sheaves and functions on the basic affine space rather than the flag variety, for consistency). Applying this to the standard sheaves $\Delta_w \in \mathcal{P}$, it follows that $\text{tr}(\Delta_w) = (-q^{-1/2})^{\ell(w)}\chi_w$, where $\chi_w$ is the indicator function on the $\mathbb{F}_q$-points of the Bruhat cell $BwU \subset G/U$. Returning to the setting where $k$ is a local field, $\mathcal{O}$ its ring of integers, $\pi$ the uniformizer, and $\kappa$ the quotient, one can construct an ``Eisenstein map" as follows.

Let $\mathbf{X}(\mathcal{O})_{\mathrm{fin}} \subset \mathbf{X}(\mathcal{O})$ denote the union of the Bruhat cells $IwU$ across all $w$ in the finite Weyl group $W$. The subset $\mathbf{X}(\mathcal{O})_{\mathrm{fin}}$ has the property that the projection $\mathcal{O} \twoheadrightarrow \kappa$ induces a well-defined and surjective map $\mathbf{X}(\mathcal{O})_{\mathrm{fin}} \to \mathbf{X}(\kappa)$. 

We then have the diagram
\[
\begin{tikzcd}
         & \mathbf{X}(\mathcal{O})_{\mathrm{fin}} \arrow[ld, "\iota"'] \arrow[rd, "p"] &               \\
\mathbf{X}(k) &                                                           & \mathbf{X}(\kappa)
\end{tikzcd}\]
which we obtain from the natural maps $k \hookleftarrow\mathcal{O} \twoheadrightarrow \kappa$. One can then consider $\iota_!\circ p^*$ as a map from the space $\mathcal{C}(\mathbf{X}(\kappa))$ of functions on $\mathbf{X}(\mathbb{\kappa})$ to $\mathcal{S}_c = \mathcal{S}_c(\mathbf{X}(k))$.

\begin{defn}
Let $\mathcal{S}_0$ be the subspace of $\mathcal{S}$ generated under the Fourier transforms $\Phi_w$ by the image of $\iota_! \circ p^*$.
\end{defn}

Then (\ref{eqn:deltadelta}) in Theorem \ref{thm:kloschwartz} tells us the following.
\begin{proposition}\label{prop:eisenstein}
The map $\Theta$ is an isomorphism onto $\bks_0$. For $\mathcal{F} \in \mathcal{P}$,
    \begin{align*}
        \Theta([j_{e!}(\mathcal{F})]) & = (\iota_! \circ p^*)(\mathrm{tr}(\mathcal{F}))\in \mathcal{S}_c^{I\times \mathbf{T}(\mathcal{O})}.
    \end{align*}
\end{proposition}

This concludes the proof of Theorem \ref{thm:schwartzconnection} from the introduction, giving a direct interpretation of the composition $K_0(\mathcal{A}_{\mathcal{P}}) \otimes \mathbb{C} \to M_{d,q} \to \mathcal{S}^{I\times \mathbf{T}(\mathcal{O})}$ (each of which is originally described combinatorially) in terms of Grothendieck's sheaf-function dictionary via the map $\iota_! \circ p^*$.

\section{Perverse sheaves on the semi-infinite flag manifold}
\subsection{Preliminaries\label{sec:semiinfprelim}}

In this section, we follow the technical setup of \cite{ABBGM}. In loc.\ cit., the authors define a category $\flinf$, intended as a model for a category of Iwahori-monodromic perverse sheaves on the semi-infinite flag variety by using the Drinfeld space $\overline{\mathrm{Bun}}_{N^-}$. (In this section, we use the notation $N^-$ for the maximal unipotent subgroup opposite $U$ to be consistent with loc.\ cit.) This construction builds off of earlier work in \cite{Flags1}, \cite{FFKM2} and \cite{BFGM}.

It is shown in \cite{ABBGM} that the category $\flinf$ has simple objects $\mathsf{IC}_{\tilde{w}}$ indexed by $\tilde{w} \in \tilde{W}$. For each such $\tilde{w}$, the standard and costandard objects $\boldsymbol{\nabla}_{\tilde{w}}$ and $\boldsymbol{\Delta}_{\tilde{w}}$ are also defined using maps $i_{\tilde{w}}$ corresponding to inclusion of the appropriate stratum indexed by $\tilde{w}$ (we use bold symbols for standard and costandard objects in $\flinf$ to distinguish them from the usual standard and costandard objects $\Delta_w$ and $\nabla_w$ in $\mathrm{Perv}_B(G/B)$; note also that their labelling as standard and costandard sheaves respectively, as opposed to the reverse, is nonstandard but follows the conventions of \cite{ABBGM}). One of the main results in loc.\ cit. is an equivalence between the category of Artinian objects in $\flinf$ and a certain category of graded modules over the small quantum group. This result categorifies a known connection between the Grothendieck group of the latter category and Lusztig's periodic Hecke module which appears in \cite{AJS}, c.f.\ Theorem 17.8. The latter category admits a $W$-action given in Lemma 1.1.16 of loc.\ cit. As a result, there is a natural $W$-action on simple objects in $\flinf$ given as follows. We note that equation (55) in Section 6.1.8 of loc.\ cit. describes this same result in terms of ``restricted irreducible" objects, but there is a typo in the indices in the equation referenced, so we reproduce a corrected version here employing instead the conventions and notation of the present paper.

\begin{proposition}[Section 6.18 in \cite{ABBGM}]
The category $\flinf$ admits a $W$-action by endofunctors $\{\mathsf{F}_w\}_{w \in W}$. These functors send simple objects to simple objects, and in particular
\begin{align*}
\mathsf{F}_{s}(\mathsf{IC}_{i(\tilde{w})}) & = \mathsf{IC}_{i(\epsilon_s(\tilde{w}))}
\end{align*}
for any $s \in S$, $\tilde{w} \in \widetilde{W}$ (c.f.\ Definition \ref{def:epsilonindices}), where $i : \waff \to \waff$ denotes the involution sending $w\cdot \lambda \to w\cdot (-\lambda)$ for $w \in W$, $\lambda \in \Gamma$.
\end{proposition}

In Section 6.13 of \cite{FFKM2}, the authors explain a direct connection between perverse sheaves on the semi-infinite flag variety and Lusztig's periodic Hecke module $M_{d}$. This is made precise for the category $\flinf$ by \cite{ABBGM}, c.f.\ Theorem 4.3.6. In 6.1.8 of loc.\ cit., the authors remark that there exists a category $\mathcal{P}^{\frac{\infty}{2}}$ consisting of ``mixed $D$-modules of Hodge-Tate type in $\flinf$" equipped with a half-Tate twist endofunctor $(\tfrac{1}{2})$ corresponding to $q^{-1/2}$ generated by the simple objects $\mathsf{IC}_{\tilde{w}}(\tfrac{m}{2})$ for all $\tilde{w} \in \tilde{W}$ and $m \in \mathbb{Z}$. As before, the corresponding Grothendieck group is then a $\mathbb{Z}[v, v^{-1}]$-module where $v^{-1}$ acts by $(\tfrac{1}{2})$. Similarly we write $\mathcal{P}^{\frac{\infty}{2}}_I$ for the analogous construction for Iwahori-equivariant objects, and $K_0(\mathcal{P}_I^{\frac{\infty}{2}}) \cong K_0(\mathcal{P}^{\frac{\infty}{2}})$. In the equivariant case when the context is clear we will use the same notation for standard, costandard, and simple objects.

Note that in the original result of \cite{FFKM2} loc.\ cit., the setup and notation of the periodic Hecke module follows the conventions of \cite{S}, which differs from those of \cite{L1} which we use in the present paper, hence the presence of the sign change on $\lambda$ in the results in this section. We also make reference to the $\fH$-action on $K_0(\flinf)$ discussed in \cite{ABBGM} and \cite{FFKM2}. Translating the result to our conventions, we arrive the following. 

\begin{theorem}[\cite{FFKM2}]\label{thm:semiinfmd}
There is an isomorphism of $\mathcal{H}_q\otimes \mathbb{C}[W]$-modules
\begin{align*}
K_0\left(\mathcal{P}^{\frac{\infty}{2}}\right) \otimes_{\mathbb{Z}[v, v^{-1}]} \mathbb{C} \to M_{d,q}
\end{align*}
such that for any $\tilde{w} = w\cdot \lambda$ (for $w \in W$, $\lambda \in \Gamma$),
\begin{align*}
[\mathsf{IC}_{\tilde{w}}] & \mapsto A_{w\cdot (-\lambda)}^\sharp,\\
[\boldsymbol{\nabla}_{\tilde{w}}] & \mapsto A_{w\cdot (-\lambda)}.
\end{align*}
\end{theorem}

\begin{defn}
Let $W_{\leq} \subset \tilde{W}$ be the set of $\tilde{w} \in \tilde{W}$ such that for any $y \in W$, $\epsilon_y(\tilde{w})$ can be written as $w \cdot \check{\nu}$ for some $w \in W$, and some coweight $\check{\nu}$ for which $\check{\nu} \leq 0$.

Let $W' \subset W_{\leq} \subset \tilde{W}$ be the set of all $\tilde{w} \in \tilde{W}$ such that there exists $w, z \in W$ with $\mathsf{IC}_{\tilde{w}} = \mathsf{F}_z(\mathsf{IC}_{i(w)})$. In other words, it is the orbit of $W$ in $\tilde{W}$ under the action defined by $i(\tilde{w}) \mapsto \epsilon_w(i(\tilde{w}))$ for $w \in W$.

Finally, let $T = W_{\leq} - W'$ be the subset of $W_{\leq}$ consisting of all elements which do not lie in $W'$.
\end{defn}

\begin{defn}
Let $\tilde{\mathcal{P}}_{\leq}$ (resp. $\mathcal{T}$) be the Serre subcategory of $\mathcal{P}^{\frac{\infty}{2}}_I$ generated by half-integer twists of $\{\mathrm{IC}_{\tilde{w}}\}_{\tilde{w}\in W_{\leq}}$ (resp. $\{\mathrm{IC}_{\tilde{w}}\}_{\tilde{w}\in T}$) under extensions. Let $\mathcal{Q}$ be the Serre quotient $\tilde{\mathcal{P}}_{\leq} / \mathcal{T}$. 

Let $\tilde{\mathcal{P}}^\circ_{\leq 0}$, $\mathcal{T}^\circ$, and $\mathcal{Q}^\circ$ be the non-mixed analogues, e.g. $\tilde{\mathcal{P}}_{\leq}^\circ$ is the Serre subcategory of $\mathrm{Perv}(\mathcal{F}\ell^\frac{\infty}{2})^{I}$ generated by $\{\mathrm{IC}_{\tilde{w}}\}_{\tilde{w} \in W_+{\leq}}$.
\end{defn}

Our aim in the subsequent sections will be to categorify Theorem \ref{thm:introklomd} by showing that $\mathcal{Q}^\circ$ and $\mathcal{A}_B$ (resp. $\mathcal{Q}$ and $\mathcal{A}_{\mathcal{P}}$) are equivalent as categories, and then showing that the Serre quotient map $\pi : \tilde{\mathcal{P}}_{\leq} \to \tilde{\mathcal{P}}_{\leq}/\mathcal{T}$ admits a fully faithful right adjoint. By composing this adjoint functor with the equivalence of categories claimed above, we will then get an equivalence of categories between $\mathcal{A}_B$ (resp. $\mathcal{A}_{\mathcal{P}}$) and a full subcategory $\tilde{\mathcal{P}}$ of $\mathcal{P}^{\frac{\infty}{2}}_I$ as claimed in Theorem \ref{thm:introcategorification}.

\subsection{Setup}

We will continue to refer to the setup and notation of \cite{ABBGM}. In loc.\ cit., the authors define the category $\mathrm{Perv}(\mathcal{F}\ell^{\frac{\infty}{2}})^{I}$ as the subcategory of sheaves in $\mathrm{Perv}({}_\infty^{I} \overline{\mathrm{Bun}}_{N^-})$ satisfying certain conditions. For any coweight $\check{\nu}$, they define ${}_{\check{\nu}}^{I} \overline{\mathrm{Bun}}_{N^-} \subset {}_\infty^{I} \overline{\mathrm{Bun}}_{N^-}$, denoting by $\overline{i}_{\check{\nu}}$ the inclusion map. We will use this map in the case where $\check{\nu} = 0$. Similarly, they define ${}_{\leq \check{\nu}}^{I} \overline{\mathrm{Bun}}_{N^-} \subset {}_\infty^{I} \overline{\mathrm{Bun}}_{N^-}$.

For any $\tilde{w} \in W'$, it is easy to see that the strata indexed by $\check{\nu}$ on which $\mathrm{IC}_{\tilde{w}}$ is supported must be such that $\check{\nu} \leq 0$. (This is the statement that $W' \subset W_{\leq}$, which follows purely from the combinatorics of $M_d$ and the definition of the operators $\theta_w$; note that the orderings $\leq$ and $\geq$ are switched when translating from $\mathrm{Perv}(\mathcal{F}\ell^{\frac{\infty}{2}})^{I}$ to $M_d$, as we explain in Section \ref{sec:semiinfprelim}.)

\begin{lemma}[\cite{ABBGM}, Section 4.1.3]
The closure of ${}_0^{I} \overline{\mathrm{Bun}}_{N^-}$ in ${}_\infty^{I} \overline{\mathrm{Bun}}_{N^-}$ is ${}_{\leq 0}^{I} \overline{\mathrm{Bun}}_{N^-}$.
\end{lemma}

\begin{corollary}\label{cor:itsexact}
The functor $\overline{i}_0^* : \tilde{\mathcal{P}}_{\leq}^\circ \to \mathrm{Perv}({}_0^{I} \overline{\mathrm{Bun}}_{N^-})$ is exact.
\end{corollary}
\begin{proof}
By the preceding discussion, any element of $\tilde{\mathcal{P}}_{\leq}^\circ$ is supported in ${}_{\leq 0}^{I} \overline{\mathrm{Bun}}_{N^-}$. It is shown in \cite{ABBGM} that ${}_0^{I} \overline{\mathrm{Bun}}_{N^-}$ is an open subset of ${}_{\leq 0}^{I} \overline{\mathrm{Bun}}_{N^-}$, from which the result follows.
\end{proof}

Now the following result, also shown in \cite{ABBGM}, will be crucial for our purposes.
\begin{proposition}[\cite{ABBGM}, Section 4]
There exists equivalences of categories
\begin{align}
    \mathrm{Perv}_U(G/B) \to {}^{'}\mathrm{Perv}({}_{0}^{I^0} \overline{\mathrm{Bun}}_{N^-})\\
    \mathrm{Perv}_B(G/B) \to {}^{'}\mathrm{Perv}({}_{0}^{I} \overline{\mathrm{Bun}}_{N^-}),
\end{align}
where ${}^{'}\mathrm{Perv}({}_{0}^{I^0}\overline{\mathrm{Bun}}_{N^-}) \subset \mathrm{Perv}({}_{0}^{I^0}\overline{\mathrm{Bun}}_{N^-})$ is the subcategory satisfying the same conditions as in the definition of $\mathrm{Perv}(\mathcal{F}\ell^\frac{\infty}{2})^{I^0} \subset \mathrm{Perv}({}_{\infty}^{I^0}\overline{\mathrm{Bun}}_{N^-})$, and similarly for the $I$-equivariant version. These equivalences send the standard sheaves $\Delta_w$ to the sheaves $\boldsymbol{\nabla}_w$ described in 4.4.3 of \cite{ABBGM}.
\end{proposition}

From now on, by means of the equivalence in this proposition, we will implicitly identify $\mathrm{Perv}_B(G/B)$ with ${}^{'}\mathrm{Perv}({}_{0}^{I}\overline{\mathrm{Bun}}_{N^-})$, viewing $\mathrm{Perv}_B(G/B)$ as the target (resp. source) of the functor $\overline{i}_0^*$ (resp. ${}^p\overline{i}_{0!}$).

\begin{prop}\label{prop:fwquotient}
    The functor $\overline{i}_0^*$ is zero when restricted to $\mathcal{T}$. Further, the category $\mathcal{T}$ is closed under the action of the $\{\mathsf{F}_w\}_{w \in W}.$ As a result, the functor $\overline{i}_0^*$ and the functors $\{\mathsf{F}_w\}_{w \in W}$ are each well-defined on the quotient category $\mathcal{Q}$.
\end{prop}

\begin{proof}
    By definition $\mathcal{T}$ is generated by a collection of irreducible objects $\mathsf{IC}_{\tilde{w}}$, each of which has the property that $\tilde{w} = w\cdot \check{\nu}$ for $w \in W$ and $\check{\nu}$ some coweight with $\check{\nu} < 0$. This means $\overline{i}_0^*\mathsf{IC}_{\tilde{w}} = 0$ for each such irreducible, and therefore $\overline{i}_0^*(\mathcal{T}) = 0$.

    Next, note that by the definition of $W'$, irreducible objects $\mathsf{IC}_{\tilde{w}}$ for $\tilde{w} \in W_{\leq} - W'$ are sent to other such irreducible objects under the functors $\{\mathsf{F}_{w}\}_{w \in W}$; i.e. these functors send irreducible objects in $\mathcal{T}$ to other irreducible objects in $\mathcal{T}$. Since these functors are exact, this shows that $\mathcal{T}$ is closed under each of them.
\end{proof}

\subsection{An equivalence of categories}\label{sec:anequiv}

In \cite[Section 1]{P}, the gluing of abelian categories via so-called ``gluing data" is described; this specializes to Kazhdan-Laumon gluing when the gluing data is given by the symplectic Fourier transforms $F_{w,!}$ and certain natural transformations between them. We will now recall a recognition theorem for glued abelian categories which appears as Theorem 1.2.1 in loc.\ cit., see also \cite{BBP} for a more general statement which expresses gluing data in the language of coalgebras and comonads.

Suppose $(\mathcal{K}_w)_{w \in W}$ is a collection of abelian categories, and $\mathcal{K}$ is an abelian category equipped with exact functors $j_{w}^* : \mathcal{K} \to \mathcal{K}_w$ for all $w \in W$. Suppose that each of these functors has a left adjoint $j_{w!} : \mathcal{K}_w \to \mathcal{K}$. For any $y, w$, we define $G_{y,w} : \mathcal{K}_w \to \mathcal{K}_y$ by $G_{y,w} = j_{y}^*j_{w!}$. For any $w,y,z \in W$, we then obtain a natural transformation of functors $G_{z,y}G_{y,w} \to G_{z,w}$ as
\begin{align}
    G_{z,y}G_{y,w} = j_{z}^*j_{y!}j_y^*j_{w!} \to j_{z}^*j_{w!}
\end{align}
arising from the counit of the adjunction $(j_{y!}, j_y^*)$. The collection of functors $(G_{y,w})_{y, w \in W}$ along with these natural transformations then automatically satisfy the conditions for being a \emph{gluing data} as defined in Definition 1.1.1 of \cite{P}.

In the case of the category $\mathcal{A}_B$ considered in the present paper, the natural transformations
\begin{align}
    {}^pF_{w,!}\circ {}^pF_{y,!} \to {}^pF_{wy,!}\label{eqn:naturals}
\end{align}
which arise in this way can be described as follows. By Proposition \ref{prop:f!conv}, it suffices to describe the corresponding morphisms
\begin{align}\label{eqn:wywy}
    \nabla_w(\tfrac{\ell(w)}{2})* \nabla_y(\tfrac{\ell(y)}{2}) \to \nabla_{wy}(\tfrac{\ell(wy)}{2}).
\end{align}
When $\ell(w) + \ell(y) = \ell(wy)$, this becomes the canonical isomorphism inherent in the braid action. There is also a natural morphism
\begin{align}
    \nabla_s(\tfrac{1}{2}) * \nabla_s(\tfrac{1}{2}) \to \Delta_e
\end{align}
which comes from applying $- * \nabla_s(\tfrac{1}{2})$ to the canonical morphism
\begin{align}
    \nabla_s(\tfrac{1}{2}) \to \Delta_e \to \Delta_s(-\tfrac{1}{2}).\label{eqn:nabdeldel}
\end{align}
Together, these two types of morphisms assemble to define morphisms as in (\ref{eqn:wywy}) for any $y, w \in W$.

We now state the recongition theorem for glued abelian categories appearing in Theorem 1.2.1 of \cite{P}, adapted to the setting of Kazhdan-Laumon categories. 

\begin{theorem}[\cite{P}]\label{thm:recognition}
    Suppose $\mathcal{K}$ is an abelian category equipped with exact functors $j_{w}^* : \mathcal{K} \to \mathrm{Perv}_B(G/B)$, each of which has a left adjoint $j_{w!}$.  Suppose that these functors satisfy $j_w^*j_{v!} = {}^pF_{wv^{-1},!}$, and that the natural transformations ${}^pF_{y,!} \circ F_{w,!} \to {}^pF_{yw,!}$ for $y, w \in W$ induced from this adjunction agree with those described in the discussion following (\ref{eqn:naturals}). Assume also that if an object $A \in \mathcal{K}$ satisfies $j_w^*A = 0$ for all $w \in W$, then $A = 0$. Then there is a natural equivalence $\mathcal{K} \cong \mathcal{A}_{B}$.
\end{theorem}

By the previous section, we know that $\overline{i}_0^*$ can be considered as a functor from $\mathcal{Q}$ to $\mathcal{P}$. Further, it is exact by Lemma \ref{cor:itsexact}. Accordingly, we make the following definition.

\begin{defn}
Let $\mathsf{j}_e^* : \mathcal{Q} \to \mathcal{P}$ be the functor $\overline{i}_0^*$. Further, for any $w \in W$, let
\begin{align*}
    \mathsf{j}_w^* & := \overline{i}_0^*\circ \mathsf{F}_{w},\\
    \mathsf{j}_{w!} & := \mathsf{F}_{w^{-1}} \circ {}^p\overline{i}_{0!}
\end{align*}
\end{defn}

\begin{prop}
    Each functor $\mathsf{j}_{w}^*$ has a left adjoint given by 
    \begin{align}
        \pi \circ \mathsf{F}_{w^{-1}} \circ {}^{p}\overline{i}_{0!}
    \end{align}
    where $\pi : \tilde{\mathcal{P}}_{\leq} \to \mathcal{Q}$ is the Serre quotient functor.
\end{prop}
\begin{proof}
    First note that for any $\mathcal{F} \in \mathcal{P}$ and $Y \in \tilde{\mathcal{P}}_{\leq}$, we have a functorial isomorphism
    \begin{align}
        \mathrm{Hom}_{\mathcal{P}}(\mathcal{F}, (\overline{i}_0^*\circ \mathsf{F}_w)(Y)) & \cong \mathrm{Hom}_{\tilde{\mathcal{P}}_{\leq}}((\mathsf{F}_{w^{-1}} \circ {}^{p}\overline{i}_{0!})(\mathcal{F}), Y)
    \end{align}
    by adjointness of $(\overline{i}_{0!}, \overline{i}_0^*)$. It remains to show that
    \begin{align}\label{eqn:bijquotmor}
        \mathrm{Hom}_{\tilde{\mathcal{P}}_{\leq}}((\mathsf{F}_{w^{-1}} \circ {}^{p}\overline{i}_{0!})(\mathcal{F}), Y) \cong \mathrm{Hom}_{\mathcal{Q}}((\pi \circ \mathsf{F}_{w^{-1}} \circ {}^{p}\overline{i}_{0!})(\mathcal{F}), \pi(Y)).
    \end{align}
    Note that $(\mathsf{F}_{w^{-1}} \circ {}^{p}\overline{i}_{0!})(\mathcal{F})$ admits no nontrivial quotient which lies in $\mathcal{T}$, by the adjointness of $(\mathsf{F}_{w^{-1}} \circ {}^{p}\overline{i}_{0!}, \mathsf{F}_{w} \circ \overline{i}_0^*)$ and by the fact that $\overline{i}_{0}^*$ is trivial on $\mathcal{T}$. By the definition of morphisms in Serre quotient categories, this means any morphism $\overline{f} \in \mathrm{Hom}_{\mathcal{Q}}((\mathsf{F}_{w^{-1}} \circ {}^{p}\overline{i}_{0!})(\mathcal{F}), \pi(Y))$ lifts to a genuine morphism $f : (\mathsf{F}_{w^{-1}} \circ {}^{p}\overline{i}_{0!})(\mathcal{F}) \to Y/Y'$, where $Y'$ is some subobject of $Y$ lying in $\mathcal{T}$. But again by the adjointness of $(\mathsf{F}_{w^{-1}} \circ {}^{p}\overline{i}_{0!}, \mathsf{F}_{w} \circ \overline{i}_0^*)$ and by the fact that $\overline{i}_{0}^*$ is trivial on $\mathcal{T}$, any such morphism lifts to a morphism from $(\mathsf{F}_{w^{-1}} \circ {}^{p}\overline{i}_{0!})(\mathcal{F})$ to $Y$, and this gives the bijection in (\ref{eqn:bijquotmor}).
\end{proof}

We know that the functors $\mathsf{j}_w^*$ are exact (since the $\mathsf{F}_w$ are exact), that $\pi \circ \mathsf{j}_{w!}$ are their left adjoints, and that
\begin{align}
\mathsf{j}_w^*\mathsf{j}_{v!} & = \overline{i}_0^*\circ \mathsf{F}_{wv^{-1}} \circ {}^{p}\overline{i}_{0!}.
\end{align}
So to show that the functors $\{\mathsf{j}_w^*\}_{w \in W}$ as defined here satisfy the conditions in Theorem \ref{thm:recognition}, it remains to show Lemmas \ref{lem:agreeswithconv} and \ref{lem:zeroifzero} below.
\begin{lemma}\label{lem:agreeswithconv}
When considered as endofunctors of $\mathcal{P}$, $\overline{i}_0^*\circ \mathsf{F}_{w} \circ {}^p\overline{i}_{0!}$ and ${}^pF_{w,!}$ are naturally isomorphic.
\end{lemma}
\begin{proof}
Recall that by Proposition \ref{prop:f!conv}, for any $\mathcal{F} \in \mathcal{P}$, we can write ${}^pF_{w,!}(\mathcal{F}) = {}^{p}H^0(\mathcal{F} * {\nabla}_w(\tfrac{\ell(w)}{2}))$. By 4.4.1 of \cite{ABBGM}, $!$-convolution on the left (which is the convention for convolution we use in the present paper) with objects of $\mathrm{Perv}_B(G/B)$ commutes with the functors $\overline{i}_{0!}$ and $\overline{i}_0^*$. By the definition of the $\mathsf{F}_w$ in loc.\ cit., it is also clear that they too commute with convolution on the left (on the level of $K_0$, this is simply the fact that the $W$ and $\mathcal{H}$-actions on the periodic Hecke module $M_d$ commute with one another). This means
\begin{align}
\overline{i}_0^*\circ \mathsf{F}_{w} \circ {}^p\overline{i}_{0!}(\mathcal{F}) & = \overline{i}_0^*\circ \mathsf{F}_{w} \circ {}^p\overline{i}_{0!}(\mathcal{F} * \Delta_{e})\\
& = {}^pH^0(\mathcal{F} * (\overline{i}_0^*\circ \mathsf{F}_{w} \circ {}^p\overline{i}_{0!}(\Delta_e))),
\end{align}
so it remains only to show that $\overline{i}_0^*\circ \mathsf{F}_{w} \circ \overline{i}_{0!}(\Delta_{e}) \cong \nabla_{w}(\tfrac{\ell(w)}{2})$. Note that since the action of $\mathsf{F}_w$ on the Grothendieck group agrees with Lusztig's $\theta_w$ operator on $M_d$,
\begin{align*}
[\Delta_w(-\tfrac{\ell(w)}{2}) * (\overline{i}_0^* \circ \mathsf{F}_{w} \circ \overline{i}_{0!}(\Delta_e))] = [\overline{i}_0^* \circ \mathsf{F}_{w} \circ \overline{i}_{0!}(\Delta_w(-\tfrac{\ell(w)}{2}))] = [\Delta_e].
\end{align*}
Since $\Delta_e$ is irreducible and the left-hand side is perverse (since $\overline{i}_{0!}(\Delta_w) = \boldsymbol{\nabla}_w$ and $\mathsf{F}_w$ is $t$-exact), this means $\Delta_w(-\tfrac{\ell(w)}{2}) * (\overline{i}_0^* \circ \mathsf{F}_{w} \circ \overline{i}_{0!}(\Delta_e)) \cong \Delta_e$, which gives a natural isomorphism $\overline{i}_0^* \circ \mathsf{F}_{w} \circ \overline{i}_{0!}(\Delta_e) \cong \nabla_w(\frac{\ell(w)}{2})$, as desired.
\end{proof}

\begin{remark}
We also note that the natural isomorphisms of functors in Lemma \ref{lem:agreeswithconv} give rise to natural transformations ${}^pF_{y,!} \circ {}^pF_{w,!} \to {}^pF_{yw,!}$ which agree with the ones described in the discussion following (\ref{eqn:naturals}). Using, as was used in the proof of Lemma \ref{lem:agreeswithconv}, the compatibility of those functors with convolution, it suffices to check this after applying the functors to the object $\Delta_{e}$. In other words, we must check that the diagram
\[
\begin{tikzcd}
(F_{y,!} \circ F_{w, !})(\Delta_{e}) \arrow[r] \arrow[d] & F_{yw,!}(\Delta_{e}) \arrow[d]\\
\nabla_{w^{-1}} * \nabla_{y^{-1}} \arrow[r] & \nabla_{w^{-1}y^{-1}}
\end{tikzcd}
\]
commutes. This computation is straightforward in the case when $\ell(yw) = \ell(y) + \ell(w)$. To handle all other cases, it suffices to check the case where $y = w = s \in S$, and then argue by induction on $\ell(y) + \ell(w)$. By adjunction, checking this case reduces to checking that the morphism
\begin{align}
    \nabla_{s} = (\overline{i}_{0}^*\circ \mathsf{F}_s \circ \overline{i}_{0!})(\Delta_{e}) \to (\overline{i}_{0}^*\circ \mathsf{F}_s \circ \overline{i}_{0*})(\Delta_{e}) = \Delta_{s}\label{eqn:nabdel}
\end{align}
obtained from the natural transformation $\overline{i}_{0!} \to \overline{i}_{0*}$ agrees with the morphism $\nabla_{s} \to \Delta_{e} \to \Delta_{s}$ as in (\ref{eqn:nabdeldel}). This is clear by the observation that the morphism (\ref{eqn:nabdel}) arises from applying $\overline{i}_{0}^*$ to the morphism
\begin{align}
    (\mathsf{F}_s \circ \overline{i}_{0!})(\Delta_e) = \boldsymbol{\nabla}_s \to \mathsf{IC}_e \to \boldsymbol{\Delta}_s = (\mathsf{F}_s \circ \overline{i}_{0*})(\Delta_e).
\end{align}
\end{remark}

\begin{lemma}\label{lem:zeroifzero}
If $A \in \mathcal{Q}$ satisfies $\mathsf{j}_w^* A = 0$ for all $w \in W$, then $A = 0$.
\end{lemma}
\begin{proof}
Note that $K_0(\mathcal{Q}) = \overline{M}_{d,q}^0$. Our discussion in Section \ref{sec:thebijection} shows that for any nonzero $B \in \overline{M}_{d,q}^0$, there exists at least one $w \in W$ for which $J_e(\theta_w(B)) \neq 0$ which
is equivalent to the desired statement under this identification.
\end{proof}

As a result, all of the conditions for Theorem \ref{thm:recognition} are satisfied, yielding the following as a corollary.

\begin{prop}\label{prop:equivtosubquot}
    There is an equivalence of categories between $\mathcal{A}_{B}$ and $\mathcal{Q}^\circ$ (resp. $\mathcal{A}_{\mathcal{P}}$ and $\mathcal{Q}$). This categorifies the morphism $\eta' : K_0(\mathcal{A}_{\mathcal{P}}) \otimes \mathbb{C} \to \overline{M}_{d,q}^0$.
\end{prop}

Now we use the following result to upgrade this result and obtain a fully faithful functor from $\mathcal{A}_B$ to $\mathcal{P}^{\frac{\infty}{2}}$.

\begin{prop}\label{prop:piadjoints}
    The Serre quotient functor $\pi : \tilde{\mathcal{P}}_{\leq} \to \mathcal{Q}$ admits a right adjoint $\tilde{\sigma}_*$ and a left adjoint $\tilde{\sigma}_{!}$, each of which is fully faithful.
\end{prop}
\begin{proof}
    Using Proposition \ref{prop:equivtosubquot}, we can let $\phi : \mathcal{Q} \to \mathcal{A}_B$ be the equivalence of categories described therein. To complete the proof of this proposition, it is enough to show that $\phi\circ \pi$ admits a right adjoint $\sigma_*$ and a left adjoint $\sigma_{!}$.

    First, we claim that for any $w \in W$, $j_{w}^*\phi\pi \mathsf{F}_{w^{-1}} = \overline{i}_{0}^*$ as functors from $\tilde{\mathcal{P}}_{\leq}^\circ \to \mathcal{P}^\circ$. Indeed, by the definition of $\phi$ we have $j_w^*\phi = \overline{i}_0^*\circ \mathsf{F}_{w}$, and so the result follows from Proposition \ref{prop:fwquotient} which says that $\mathsf{F}_w$ descends to a map on $\mathcal{Q}$ and therefore commutes with $\pi$. 
    
    By the dual result to Lemma 0793 in \cite{stacks-project}, it is enough to show that there exists a subset $\mathcal{J} \subset \mathrm{Ob}(\mathcal{A}_{B})$ such that
    \begin{enumerate}
        \item \label{item:subobj} Every object of $\mathcal{A}_{B}$ is a subobject of an object in $\mathcal{J}$,
        \item \label{item:choice} For every $J \in \mathcal{J}$ there exists a choice of object $\sigma_*(J)$ in $\tilde{\mathcal{P}}_{\leq}$ such that $\mathrm{Hom}_{\tilde{\mathcal{P}}_{\leq}}(X, \sigma_*(J)) \cong \mathrm{Hom}_{\mathcal{A}_B}(\phi(X), J)$ functorially for every $X$ in $\tilde{\mathcal{P}}_{\leq}$.
    \end{enumerate}
    Let $\mathcal{J} \subset \mathrm{Ob}(\mathcal{A}_B)$ be the collection of objects of the form $j_{w*}(\mathcal{F})$ for $w \in W$, $\mathcal{F} \in \mathcal{P}^\circ$. First note that (\ref{item:subobj}) holds because for any $A \in \mathcal{A}_B$, the adjunctions $(j_{w}^*, j_{w*})$ give a morphism
    \begin{align}
        A \to \bigoplus_{w \in W} j_{w_*}j_w^*A
    \end{align}
    which is easily seen to be an inclusion.

    Now to meet condition (\ref{item:choice}), for any object of the form $j_{w*}(\mathcal{F})$, we let $\sigma_*(j_{w*}(\mathcal{F})) = \mathsf{F}_{w^{-1}} ({}^p\overline{i}_{0*}(\mathcal{F}))$. Then for any $X \in \tilde{\mathcal{P}}_{\leq}$, 
    \begin{align*}
        \mathrm{Hom}_{\tilde{\mathcal{P}}_{\leq}}(X, \sigma_*(j_{w*}(\mathcal{F}))) & = \mathrm{Hom}_{\tilde{\mathcal{P}}_{\leq}}(X, \mathsf{F}_{w^{-1}}({}^p\overline{i}_{0*}(\mathcal{F})))\\
        & = \mathrm{Hom}_{\tilde{\mathcal{P}}_{\leq}}(\mathsf{F}_{w}(X), {}^p\overline{i}_{0*}(\mathcal{F}))\\
        & = \mathrm{Hom}_{\mathcal{P}}(\overline{i}_0^*(\mathsf{F}_{w}(X)), \mathcal{F})\\
        & = \mathrm{Hom}_{\mathcal{P}}(j_w^*\phi \pi(X), \mathcal{F})\\
        & = \mathrm{Hom}_{\mathcal{A}_B}((\phi\circ \pi)(X), j_{w*}(\mathcal{F}))
    \end{align*}
    with the step from the third to the fourth line using the fact that $j_w^*\phi\pi\mathsf{F}_{w^{-1}} = \overline{i}_0^*$ as functors from $\mathcal{P}_{\leq 0}$ to $\mathcal{P}$, with each of the isomorphisms in the above being functorial.

    This means the assignment $\sigma_*$ above extends to a functor $\sigma_* : \mathcal{A}_B \to \tilde{\mathcal{P}}_{\leq}$ which is right-adjoint to the Serre quotient functor $\pi$. Since the quotient functor $\pi$ is a localization functor and $\sigma_*$ is its right adjoint, we then have that $\sigma$ is necessarily fully faithful.

    Finally, we construct a left adjoint $\sigma_{!}$ to $\phi \pi$. By the same result as used above, we can repeat the above argument with quotients rather than subobjects, using the fact that for any $A \in \mathcal{A}_B$, there is a surjective morphism
    \begin{align}
        \oplus_{w \in W} j_{w!}j_{w}^*A \to A.
    \end{align}
    We then define $\sigma_{!}(j_{w!}(\mathcal{F})) = \mathsf{F}_{w^{-1}}({}^p\overline{i}_{0!}(\mathcal{F}))$. Now to meet the dual version of condition (\ref{item:choice}) above we proceed with the same argument but dualized, and conclude that indeed $\sigma_{!}$ is left-adjoint to $\phi\pi$.
\end{proof}

\begin{defn}
    Let $\tilde{\mathcal{P}}$ be the image of the functor $\tilde{\sigma}_!$, which is a full subcategory of $\mathcal{P}^\frac{\infty}{2}_I$ by Proposition \ref{prop:piadjoints}.
\end{defn}

As a corollary to Proposition \ref{prop:piadjoints}, we get the following elaboration on Theorem \ref{thm:introcategorification}. Note that the fact that the equivalence of categories induced by the left derived functor to the fully faithful right-exact functor $\sigma_!$ described in Proposition \ref{prop:piadjoints} categorifies the isomorphism $\eta$ follows from the fact that $\sigma_!(j_{e!}(\Delta_w)) = \boldsymbol{\nabla}_w$. Its compatibility with convolution follows from the fact that ${}^p\overline{i}_{0!}$ commutes with $!$-convolution by 4.4.1 of \cite{ABBGM}.

\begin{theorem}
There exists an equivalence of categories between $\mathcal{A}_{B}$ and $\tilde{\mathcal{P}}^\circ$ (resp. $\kl$ and $\tilde{\mathcal{P}}$). This categorifies the isomorphism $\eta : K_0(\kl)\otimes \mathbb{C} \to M_d^0$, and respects convolution on the left by the standard and costandard objects $\Delta_w$ and $\nabla_w$ of $\mathrm{Perv}_{B,\mathrm{m}}(G/B)$ for all $w \in W$. The $W$-action on $\kl$ by functors $\mathcal{F}_w$ is mapped to the $W$-action on $\tilde{\mathcal{P}}$ by the functors $\mathsf{F}_w$.
\end{theorem}

\bibliographystyle{alpha}
\bibliography{bibl}

\clearpage

\appendix

\section{Illustrations for $\overline{M}_d^0$ and $K_0(\mathcal{A}_B)$ in rank 2}

Expanding on Figure \ref{fig:hexagons} but suppressing the labels, we now provide illustrations in rank $2$ for $\overline{M}_d^0$ and $K_0(\mathcal{A}_B)$. More precisely, in the figures below, the shaded alcoves are the alcoves $A \in \Xi$ for which $\overline{A}^\sharp \in \overline{M}_d^0$. By Theorem \ref{thm:klophm}, the shaded alcoves index a basis for $K_0(\mathcal{A}_B)$. The colors indicate the orbits of the $\overline{A}^\sharp$ in $\overline{M}_d^0$ under the operators $\{\theta_w\}_{w \in W}$, which correspond to the orbits of simple objects in $\mathcal{A}_B$ under the operators $\{\mathcal{F}_w\}_{w \in W}$. As in Figure \ref{fig:hexagons}, the alcove $A_e$ is always pictured in red, and the fundamental alcoves $\{A_w\}_{w \in W}$ are outlined with a bold stroke. (To draw these pictures, we used Proposition \ref{prop:actionsharps} which provides a description, in terms of alcoves and reflections along root hyperplanes, of the orbits of $\{A_w^\sharp\}_{w \in W}$ under the operators $\{\theta_w\}_{w \in W}$.)

By Theorem \ref{thm:klophm}, the number of shaded alcoves in each example is the number of simple objects in $\mathcal{A}_B$. We hope that these rank $2$ illustrations demonstrate the underlying alcove geometry behind the formula in Corollary \ref{prop:counting}.

\begin{figure}[ht]
    \centering
    \caption{Type $\mathbf{A}_2$}
    \begin{tikzpicture}[scale=1.3]
    \path (0, {5.2*sqrt(3)}) -- (5, {5.2*sqrt(3)});
    
    \draw [fill=color1] (2, {2*sqrt(3)}) -- (4, {2*sqrt(3)}) -- (3, {3*sqrt(3)});
    \draw [fill=color2] (1, {3*sqrt(3)}) -- (3, {3*sqrt(3)}) -- (2, {2*sqrt(3)});
    \draw [fill=color3] (0, {2*sqrt(3)}) -- (2, {2*sqrt(3)}) -- (1, {3*sqrt(3)});
    \draw [fill=color4] (0, {2*sqrt(3)}) -- (2, {2*sqrt(3)}) -- (1, {1*sqrt(3)});
    \draw [fill=color5] (1, {1*sqrt(3)}) -- (3, {1*sqrt(3)}) -- (2, {2*sqrt(3)});
    \draw [fill=color6] (2, {2*sqrt(3)}) -- (4, {2*sqrt(3)}) -- (3, {1*sqrt(3)});
    \draw [fill=color4] (0, {4*sqrt(3)}) -- (2, {4*sqrt(3)}) -- (1, {3*sqrt(3)});
    \draw [fill=color5] (1, {3*sqrt(3)}) -- (3, {3*sqrt(3)}) -- (2, {4*sqrt(3)});
    \draw [fill=color6] (2, {4*sqrt(3)}) -- (4, {4*sqrt(3)}) -- (3, {3*sqrt(3)});
    \draw [fill=color3] (3, {3*sqrt(3)}) -- (5, {3*sqrt(3)}) -- (4, {4*sqrt(3)});
    \draw [fill=color2] (4, {4*sqrt(3)}) -- (6, {4*sqrt(3)}) -- (5, {3*sqrt(3)});
    \draw [fill=color4] (6, {4*sqrt(3)}) -- (8, {4*sqrt(3)}) -- (7, {3*sqrt(3)});
    \draw [fill=color4] (6, {2*sqrt(3)}) -- (8, {2*sqrt(3)}) -- (7, {1*sqrt(3)});
    \draw [fill=color4] (3, {5*sqrt(3)}) -- (5, {5*sqrt(3)}) -- (4, {4*sqrt(3)});
    \draw [fill=color4] (3, {1*sqrt(3)}) -- (5, {1*sqrt(3)}) -- (4, {0*sqrt(3)});
    \draw [fill=color2] (4, {2*sqrt(3)}) -- (6, {2*sqrt(3)}) -- (5, {1*sqrt(3)});
    \draw [fill=color6] (5, {3*sqrt(3)}) -- (7, {3*sqrt(3)}) -- (6, {2*sqrt(3)});
    \draw [fill=color3] (3, {1*sqrt(3)}) -- (5, {1*sqrt(3)}) -- (4, {2*sqrt(3)});
    \draw [fill=color5] (4, {2*sqrt(3)}) -- (6, {2*sqrt(3)}) -- (5, {3*sqrt(3)});
    
    \draw (0,0) -- (8,0);
    \draw (0,{sqrt(3)}) -- (8,{sqrt(3)});
    \draw (0,{2*sqrt(3)}) -- (8,{2*sqrt(3)});
    \draw (0,{3*sqrt(3)}) -- (8,{3*sqrt(3)});
    \draw (0,{4*sqrt(3)}) -- (8,{4*sqrt(3)});
    \draw (0,{5*sqrt(3)}) -- (8,{5*sqrt(3)});
    \draw (0,{4*sqrt(3)}) -- (1, {5*sqrt(3)});
    \draw (0,{2*sqrt(3)}) -- (3, {5*sqrt(3)});
    \draw (0,0) -- (5, {5*sqrt(3)});
    \draw (2,0) -- (7, {5*sqrt(3)});
    \draw (4,0) -- (8, {4*sqrt(3)});
    \draw (6,0) -- (8, {2*sqrt(3)});
    \draw (0, {2*sqrt(3)}) -- (2, 0);
    \draw (0, {4*sqrt(3)}) -- (4, 0);
    \draw (1, {5*sqrt(3)}) -- (6, 0);
    \draw (3, {5*sqrt(3)}) -- (8, 0);
    \draw (5, {5*sqrt(3)}) -- (8, {2*sqrt(3)});
    \draw (7, {5*sqrt(3)}) -- (8, {4*sqrt(3)});
    \draw [line width=0.5mm] (1, {sqrt(3)}) -- (3, {sqrt(3)});
    \draw [line width=0.5mm] (1, {3*sqrt(3)}) -- (3, {3*sqrt(3)});
    \draw [line width=0.5mm] (3, {3*sqrt(3)}) -- (4, {2*sqrt(3)});
    \draw [line width=0.5mm] (0, {2*sqrt(3)}) -- (1, {3*sqrt(3)});
    \draw [line width=0.5mm] (3, {sqrt(3)}) -- (4, {2*sqrt(3)});
    \draw [line width=0.5mm] (0, {2*sqrt(3)}) -- (1, {sqrt(3)});
    \end{tikzpicture}
\end{figure}
\begin{figure}[ht]
\caption{Type $\mathbf{B}_2$}
    \centering
    
    \begin{tikzpicture}[scale=0.88]
    \path (0, {10.5}) -- (10, {10.5});
    \draw (0,0) -- (10, 0);
    \draw (0,2) -- (10, 2);
    \draw (0,4) -- (10, 4);
    \draw (0,6) -- (10, 6);
    \draw (0,8) -- (10, 8);
    \draw (0,10) -- (10, 10);
    \draw (0,0) -- (0,10);
    \draw (2,0) -- (2,10);
    \draw (4,0) -- (4,10);
    \draw (6,0) -- (6,10);
    \draw (8,0) -- (8,10);
    \draw (10,0) -- (10,10);
    \draw (0,8) -- (2,10);
    \draw (0,6) -- (4,10);
    \draw (0,4) -- (6,10);
    \draw (0,2) -- (8,10);
    \draw (0,0) -- (10,10);
    \draw (2,0) -- (10,8);
    \draw (4,0) -- (10,6);
    \draw (6,0) -- (10,4);
    \draw (8,0) -- (10,2);
    \draw (0,2) -- (2,0);
    \draw (0,4) -- (4,0);
    \draw (0,6) -- (6,0);
    \draw (0,8) -- (8,0);
    \draw (0,10) -- (10,0);
    \draw (2,10) -- (10,2);
    \draw (4,10) -- (10,4);
    \draw (6,10) -- (10,6);
    \draw (8,10) -- (10,8);
    \draw [fill=color1] (2, 4) -- (3, 5) -- (4,4) -- (2, 4);
    \draw [fill=color2] (2,4) -- (2,6) -- (3,5) -- (2,4);
    \draw [fill=color5] (1,5) -- (2,6) -- (2,4) -- (1, 5);
    \draw [fill=color4] (0,4) -- (1, 5) -- (2,4) -- (0,4);
    \draw [fill=color2] (4,2) -- (4,4) -- (5,3) -- (4,2);
    \draw [fill=color5] (3,3) -- (4,4) -- (4,2) -- (3, 3);
    \draw [fill=color4] (2,2) -- (3, 3) -- (4,2) -- (2,2);
    \draw [fill=color2] (4,6) -- (4,8) -- (5,7) -- (4,6);
    \draw [fill=color5] (3,7) -- (4,8) -- (4,6) -- (3, 7);
    \draw [fill=color4] (2,6) -- (3, 7) -- (4,6) -- (2,6);
    \draw [fill=color2] (6,4) -- (6,6) -- (7,5) -- (6,4);
    \draw [fill=color5] (5,5) -- (6,6) -- (6,4) -- (5, 5);
    \draw [fill=color4] (4,4) -- (5, 5) -- (6,4) -- (4,4);
    \draw [fill=color3] (1,3) -- (2,2) -- (2,4) -- (1,3);
    \draw [fill=color8] (2,2) -- (2,4) -- (3,3) -- (2,2);
    \draw [fill=color7] (2,4) -- (4,4) -- (3,3) -- (2,4);
    \draw [fill=color3] (1,7) -- (2,6) -- (2,8) -- (1,7);
    \draw [fill=color8] (2,6) -- (2,8) -- (3,7) -- (2,6);
    \draw [fill=color7] (2,8) -- (4,8) -- (3,7) -- (2,8);
    \draw [fill=color3] (5,3) -- (6,2) -- (6,4) -- (5,3);
    \draw [fill=color8] (6,2) -- (6,4) -- (7,3) -- (6,2);
    \draw [fill=color7] (6,4) -- (8,4) -- (7,3) -- (6,4);
    \draw [fill=color3] (5,7) -- (6,6) -- (6,8) -- (5,7);
    \draw [fill=color8] (6,6) -- (6,8) -- (7,7) -- (6,6);
    \draw [fill=color7] (6,8) -- (8,8) -- (7,7) -- (6,8);
    \draw [fill=color6] (0,4) -- (2,4) -- (1,3) -- (0,4);
    \draw [fill=color6] (0,8) -- (2,8) -- (1,7) -- (0,8);
    \draw [fill=color6] (2,10) -- (4,10) -- (3,9) -- (2,10);
    \draw [fill=color6] (6,10) -- (8,10) -- (7,9) -- (6,10);
    \draw [fill=color6] (2,2) -- (4,2) -- (3,1) -- (2,2);
    \draw [fill=color6] (6,2) -- (8,2) -- (7,1) -- (6,2);
    \draw [fill=color6] (8,4) -- (10,4) -- (9,3) -- (8,4);
    \draw [fill=color6] (8,8) -- (10,8) -- (9,7) -- (8,8);
    \draw [line width=0.5mm] (0,4) -- (2,2) -- (4, 4) -- (2, 6) -- (0,4);
    \end{tikzpicture}
\end{figure}

\begin{figure}[ht]
    \centering
    \caption{Type $\mathbf{G}_2$}
    \begin{tikzpicture}[scale=0.744]
    
    \path (0, {(6.2)*sqrt(3)}) -- (1, {(6.2)*sqrt(3)});
    \draw [fill=color1] (2,{2*sqrt(3)}) -- (3,{2*sqrt(3)}) -- (3, {(2 + 1/3)*sqrt(3)});
    \draw [fill=color12] (2,{2*sqrt(3)}) -- (3,{(2 + 1/3)*sqrt(3)}) -- (2.5, {(2.5)*sqrt(3)});
    \draw [fill=color2] (2,{2*sqrt(3)}) -- (2.5, {(2.5)*sqrt(3)}) -- (2,{(2 + 2/3)*sqrt(3)});
    \draw [fill=color5] (2,{2*sqrt(3)}) -- (1.5,{(2 + 1/2)*sqrt(3)}) -- (2, {(2 + 2/3)*sqrt(3)});
    \draw [fill=color6] (2,{2*sqrt(3)}) -- (1.5,{(2 + 1/2)*sqrt(3)}) -- (1, {(2 + 1/3)*sqrt(3)});
    \draw [fill=color4] (2,{2*sqrt(3)}) -- (1,{(2 + 0/2)*sqrt(3)}) -- (1, {(2 + 1/3)*sqrt(3)});
    
    \draw [fill=color10] (2,{2*sqrt(3)}) -- (1,{(2)*sqrt(3)}) -- (1, {(2 - 1/3)*sqrt(3)});
    
    \draw [fill=color11] (2,{2*sqrt(3)}) -- (1.5,{(1.5)*sqrt(3)}) -- (1, {(2 - 1/3)*sqrt(3)});
    \draw [fill=color7] (2,{2*sqrt(3)}) -- (1.5,{(1.5)*sqrt(3)}) -- (2, {(2 - 2/3)*sqrt(3)});
    \draw [fill=color8] (2,{2*sqrt(3)}) -- (2.5,{(1.5)*sqrt(3)}) -- (2, {(2 - 2/3)*sqrt(3)});
    \draw [fill=color3] (2,{2*sqrt(3)}) -- (2.5,{(1.5)*sqrt(3)}) -- (3, {(2 - 1/3)*sqrt(3)});
    \draw [fill=color9] (2,{2*sqrt(3)}) -- (3,{(2)*sqrt(3)}) -- (3, {(2 - 1/3)*sqrt(3)});

    \draw [fill=color12] (3,{3*sqrt(3)}) -- (4,{(3 + 1/3)*sqrt(3)}) -- (3.5, {(3.5)*sqrt(3)});
    \draw [fill=color2] (3,{3*sqrt(3)}) -- (3.5, {(3.5)*sqrt(3)}) -- (3,{(3 + 2/3)*sqrt(3)});
    \draw [fill=color5] (3,{3*sqrt(3)}) -- (2.5,{(3 + 1/2)*sqrt(3)}) -- (3, {(3 + 2/3)*sqrt(3)});
    \draw [fill=color6] (3,{3*sqrt(3)}) -- (2.5,{(3 + 1/2)*sqrt(3)}) -- (2, {(3 + 1/3)*sqrt(3)});
    \draw [fill=color4] (3,{3*sqrt(3)}) -- (2,{(3 + 0/2)*sqrt(3)}) -- (2, {(3 + 1/3)*sqrt(3)});

    \draw [fill=color12] (5,{3*sqrt(3)}) -- (6,{(3 + 1/3)*sqrt(3)}) -- (5.5, {(3.5)*sqrt(3)});
    \draw [fill=color2] (5,{3*sqrt(3)}) -- (5.5, {(3.5)*sqrt(3)}) -- (5,{(3 + 2/3)*sqrt(3)});
    \draw [fill=color5] (5,{3*sqrt(3)}) -- (4.5,{(3 + 1/2)*sqrt(3)}) -- (5, {(3 + 2/3)*sqrt(3)});
    \draw [fill=color6] (5,{3*sqrt(3)}) -- (4.5,{(3 + 1/2)*sqrt(3)}) -- (4, {(3 + 1/3)*sqrt(3)});
    \draw [fill=color4] (5,{3*sqrt(3)}) -- (4,{(3 + 0/2)*sqrt(3)}) -- (4, {(3 + 1/3)*sqrt(3)});

    \draw [fill=color12] (3,{1*sqrt(3)}) -- (4,{(1 + 1/3)*sqrt(3)}) -- (3.5, {(1.5)*sqrt(3)});
    \draw [fill=color2] (3,{1*sqrt(3)}) -- (3.5, {(1.5)*sqrt(3)}) -- (3,{(1 + 2/3)*sqrt(3)});
    \draw [fill=color5] (3,{1*sqrt(3)}) -- (2.5,{(1 + 1/2)*sqrt(3)}) -- (3, {(1 + 2/3)*sqrt(3)});
    \draw [fill=color6] (3,{1*sqrt(3)}) -- (2.5,{(1 + 1/2)*sqrt(3)}) -- (2, {(1 + 1/3)*sqrt(3)});
    \draw [fill=color4] (3,{1*sqrt(3)}) -- (2,{(1 + 0/2)*sqrt(3)}) -- (2, {(1 + 1/3)*sqrt(3)});

    \draw [fill=color12] (5,{1*sqrt(3)}) -- (6,{(1 + 1/3)*sqrt(3)}) -- (5.5, {(1.5)*sqrt(3)});
    \draw [fill=color2] (5,{1*sqrt(3)}) -- (5.5, {(1.5)*sqrt(3)}) -- (5,{(1 + 2/3)*sqrt(3)});
    \draw [fill=color5] (5,{1*sqrt(3)}) -- (4.5,{(1 + 1/2)*sqrt(3)}) -- (5, {(1 + 2/3)*sqrt(3)});
    \draw [fill=color6] (5,{1*sqrt(3)}) -- (4.5,{(1 + 1/2)*sqrt(3)}) -- (4, {(1 + 1/3)*sqrt(3)});
    \draw [fill=color4] (5,{1*sqrt(3)}) -- (4,{(1 + 0/2)*sqrt(3)}) -- (4, {(1 + 1/3)*sqrt(3)});

    \draw [fill=color12] (6,{2*sqrt(3)}) -- (7,{(2 + 1/3)*sqrt(3)}) -- (6.5, {(2.5)*sqrt(3)});
    \draw [fill=color2] (6,{2*sqrt(3)}) -- (6.5, {(2.5)*sqrt(3)}) -- (6,{(2 + 2/3)*sqrt(3)});
    \draw [fill=color5] (6,{2*sqrt(3)}) -- (5.5,{(2 + 1/2)*sqrt(3)}) -- (6, {(2 + 2/3)*sqrt(3)});
    \draw [fill=color6] (6,{2*sqrt(3)}) -- (5.5,{(2 + 1/2)*sqrt(3)}) -- (5, {(2 + 1/3)*sqrt(3)});
    \draw [fill=color4] (6,{2*sqrt(3)}) -- (5,{(2 + 0/2)*sqrt(3)}) -- (5, {(2 + 1/3)*sqrt(3)});

    \draw [fill=color11] (2,{4*sqrt(3)}) -- (1.5,{(3.5)*sqrt(3)}) -- (1, {(4 - 1/3)*sqrt(3)});
    \draw [fill=color7] (2,{4*sqrt(3)}) -- (1.5,{(3.5)*sqrt(3)}) -- (2, {(4 - 2/3)*sqrt(3)});
    \draw [fill=color8] (2,{4*sqrt(3)}) -- (2.5,{(3.5)*sqrt(3)}) -- (2, {(4 - 2/3)*sqrt(3)});
    \draw [fill=color3] (2,{4*sqrt(3)}) -- (2.5,{(3.5)*sqrt(3)}) -- (3, {(4 - 1/3)*sqrt(3)});
    \draw [fill=color9] (2,{4*sqrt(3)}) -- (3,{(4)*sqrt(3)}) -- (3, {(4 - 1/3)*sqrt(3)});

    \draw [fill=color11] (8,{4*sqrt(3)}) -- (7.5,{(3.5)*sqrt(3)}) -- (7, {(4 - 1/3)*sqrt(3)});
    \draw [fill=color7] (8,{4*sqrt(3)}) -- (7.5,{(3.5)*sqrt(3)}) -- (8, {(4 - 2/3)*sqrt(3)});
    \draw [fill=color8] (8,{4*sqrt(3)}) -- (8.5,{(3.5)*sqrt(3)}) -- (8, {(4 - 2/3)*sqrt(3)});
    \draw [fill=color3] (8,{4*sqrt(3)}) -- (8.5,{(3.5)*sqrt(3)}) -- (9, {(4 - 1/3)*sqrt(3)});
    \draw [fill=color9] (8,{4*sqrt(3)}) -- (9,{(4)*sqrt(3)}) -- (9, {(4 - 1/3)*sqrt(3)});
    
    \draw [fill=color11] (8,{2*sqrt(3)}) -- (7.5,{(1.5)*sqrt(3)}) -- (7, {(2 - 1/3)*sqrt(3)});
    \draw [fill=color7] (8,{2*sqrt(3)}) -- (7.5,{(1.5)*sqrt(3)}) -- (8, {(2 - 2/3)*sqrt(3)});
    \draw [fill=color8] (8,{2*sqrt(3)}) -- (8.5,{(1.5)*sqrt(3)}) -- (8, {(2 - 2/3)*sqrt(3)});
    \draw [fill=color3] (8,{2*sqrt(3)}) -- (8.5,{(1.5)*sqrt(3)}) -- (9, {(2 - 1/3)*sqrt(3)});
    \draw [fill=color9] (8,{2*sqrt(3)}) -- (9,{(2)*sqrt(3)}) -- (9, {(2 - 1/3)*sqrt(3)});

    \draw [fill=color11] (5,{5*sqrt(3)}) -- (4.5,{(4.5)*sqrt(3)}) -- (4, {(5 - 1/3)*sqrt(3)});
    \draw [fill=color7] (5,{5*sqrt(3)}) -- (4.5,{(4.5)*sqrt(3)}) -- (5, {(5 - 2/3)*sqrt(3)});
    \draw [fill=color8] (5,{5*sqrt(3)}) -- (5.5,{(4.5)*sqrt(3)}) -- (5, {(5 - 2/3)*sqrt(3)});
    \draw [fill=color3] (5,{5*sqrt(3)}) -- (5.5,{(4.5)*sqrt(3)}) -- (6, {(5 - 1/3)*sqrt(3)});
    \draw [fill=color9] (5,{5*sqrt(3)}) -- (6,{(5)*sqrt(3)}) -- (6, {(5 - 1/3)*sqrt(3)});

    \draw [fill=color11] (5,{1*sqrt(3)}) -- (4.5,{(0.5)*sqrt(3)}) -- (4, {(1 - 1/3)*sqrt(3)});
    \draw [fill=color7] (5,{1*sqrt(3)}) -- (4.5,{(0.5)*sqrt(3)}) -- (5, {(1 - 2/3)*sqrt(3)});
    \draw [fill=color8] (5,{1*sqrt(3)}) -- (5.5,{(0.5)*sqrt(3)}) -- (5, {(1 - 2/3)*sqrt(3)});
    \draw [fill=color3] (5,{1*sqrt(3)}) -- (5.5,{(0.5)*sqrt(3)}) -- (6, {(1 - 1/3)*sqrt(3)});
    \draw [fill=color9] (5,{1*sqrt(3)}) -- (6,{(1)*sqrt(3)}) -- (6, {(1 - 1/3)*sqrt(3)});

    \draw [fill=color10] (2,{4*sqrt(3)}) -- (1,{(4)*sqrt(3)}) -- (1, {(4 - 1/3)*sqrt(3)});
    \draw [fill=color10] (2,{2*sqrt(3)}) -- (1,{(2)*sqrt(3)}) -- (1, {(2 - 1/3)*sqrt(3)});
    \draw [fill=color10] (3,{5*sqrt(3)}) -- (2,{(5)*sqrt(3)}) -- (2, {(5 - 1/3)*sqrt(3)});
    \draw [fill=color10] (6,{6*sqrt(3)}) -- (5,{(6)*sqrt(3)}) -- (5, {(6 - 1/3)*sqrt(3)});
    \draw [fill=color10] (8,{6*sqrt(3)}) -- (7,{(6)*sqrt(3)}) -- (7, {(6 - 1/3)*sqrt(3)});
    \draw [fill=color10] (6,{0*sqrt(3)}) -- (5,{(0)*sqrt(3)}) -- (5, {(0 - 1/3)*sqrt(3)});
    \draw [fill=color10] (8,{0*sqrt(3)}) -- (7,{(0)*sqrt(3)}) -- (7, {(0 - 1/3)*sqrt(3)});
    \draw [fill=color10] (12,{4*sqrt(3)}) -- (11,{(4)*sqrt(3)}) -- (11, {(4 - 1/3)*sqrt(3)});
    \draw [fill=color10] (12,{2*sqrt(3)}) -- (11,{(2)*sqrt(3)}) -- (11, {(2 - 1/3)*sqrt(3)});
    \draw [fill=color10] (3,{1*sqrt(3)}) -- (2,{(1)*sqrt(3)}) -- (2, {(1 - 1/3)*sqrt(3)});
    \draw [fill=color10] (3,{5*sqrt(3)}) -- (2,{(5)*sqrt(3)}) -- (2, {(5 - 1/3)*sqrt(3)});
    \draw [fill=color10] (11,{1*sqrt(3)}) -- (10,{(1)*sqrt(3)}) -- (10, {(1 - 1/3)*sqrt(3)});
    \draw [fill=color10] (11,{5*sqrt(3)}) -- (10,{(5)*sqrt(3)}) -- (10, {(5 - 1/3)*sqrt(3)});

    
    \draw (0,0) -- (12,0);
    \draw (0,{sqrt(3)}) -- (12,{sqrt(3)});
    \draw (0,{2*sqrt(3)}) -- (12,{2*sqrt(3)});
    \draw (0,{3*sqrt(3)}) -- (12,{3*sqrt(3)});
    \draw (0,{4*sqrt(3)}) -- (12,{4*sqrt(3)});
    \draw (0,{5*sqrt(3)}) -- (12,{5*sqrt(3)});
    \draw (0,{6*sqrt(3)}) -- (12,{6*sqrt(3)});
    \draw (0,{-sqrt(3)}) -- (12,{-sqrt(3)});
    \draw (0,{4*sqrt(3)}) -- (2, {6*sqrt(3)});
    \draw (0,{2*sqrt(3)}) -- (4, {6*sqrt(3)});
    \draw (0,0) -- (6, {6*sqrt(3)});
    \draw (1,{-sqrt(3)}) -- (8, {6*sqrt(3)});
    \draw (3,{-sqrt(3)}) -- (10, {6*sqrt(3)});
    \draw (5,{-sqrt(3)}) -- (12, {6*sqrt(3)});
    \draw (7,{-sqrt(3)}) -- (12, {4*sqrt(3)});
    \draw (9, {-sqrt(3)}) -- (12, {2*sqrt(3)});
    \draw (11, {-sqrt(3)}) -- (12, {0*sqrt(3)});
    \draw (0, {2*sqrt(3)}) -- (3, {-sqrt(3)});
    \draw (0, {4*sqrt(3)}) -- (5, {-sqrt(3)});
    \draw (0, {6*sqrt(3)}) -- (7, {-sqrt(3)});
    \draw (2, {6*sqrt(3)}) -- (9, {-sqrt(3)});
    \draw (4, {6*sqrt(3)}) -- (11, {-sqrt(3)});
    \draw (6, {6*sqrt(3)}) -- (12, {0*sqrt(3)});
    \draw (8, {6*sqrt(3)}) -- (12, {2*sqrt(3)});
    \draw (10, {6*sqrt(3)}) -- (12, {4*sqrt(3)});
    \draw (0, 0) -- (1, {-sqrt(3)});
    \draw (0,0) -- (12,{(4*sqrt(3) + 0*sqrt(3)/3});
    \draw (0,{2*sqrt(3)}) -- (12,{(5*sqrt(3) + 3*sqrt(3)/3});
    \draw (0,{4*sqrt(3)}) -- (6,{6*sqrt(3)});
    \draw (3,{-sqrt(3)}) -- (12,{(1 + 3/3)*sqrt(3)});
    \draw (5,{-sqrt(3)}) -- (12,{(0 + 4/3)*sqrt(3)});
    \draw (7,{-sqrt(3)}) -- (12,{(0 + 2/3)*sqrt(3)});
    \draw (9,{-sqrt(3)}) -- (12,{(-1 + 3/3)*sqrt(3)});
    \draw (11,{-sqrt(3)}) -- (12,{(-1 + 1/3)*sqrt(3)});
    \draw (0,{(-1 + 1/3)*sqrt(3)}) -- (12,{(3 + 1/3)*sqrt(3)});
    \draw (0,{(0 + 2/3)*sqrt(3)}) -- (12,{(4 + 2/3)*sqrt(3)});
    \draw (0,{(1 + 1/3)*sqrt(3)}) -- (12,{(4 + 4/3)*sqrt(3)});
    \draw (0,{(2 + 2/3)*sqrt(3)}) -- (10,{(6)*sqrt(3)});
    \draw (0,{(3 + 1/3)*sqrt(3)}) -- (8,{(6)*sqrt(3)});
    \draw (0,{(4 + 2/3)*sqrt(3)}) -- (4,{(6)*sqrt(3)});
    \draw (0,{(5 + 1/3)*sqrt(3)}) -- (2,{(6)*sqrt(3)});
    \draw (1,{-sqrt(3)}) -- (12,{(2 + 2/3)*sqrt(3)});
    \draw (0,{(-1 + 1/3)*sqrt(3)}) -- (1,{(-1)*sqrt(3)});
    \draw (0,{0}) -- (3,{(-1)*sqrt(3)});
    \draw (0,{(2/3)*sqrt(3)}) -- (5,{(-1)*sqrt(3)});
    \draw (0,{(1 + 1/3)*sqrt(3)}) -- (7,{(-1)*sqrt(3)});
    \draw (0,{(1 + 3/3)*sqrt(3)}) -- (9,{(-1)*sqrt(3)});
    \draw (0,{(2 + 2/3)*sqrt(3)}) -- (11,{(-1)*sqrt(3)});
    \draw (0,{(3 + 1/3)*sqrt(3)}) -- (12,{(-1 + 1/3)*sqrt(3)});
    \draw (0,{(4)*sqrt(3)}) -- (12,{(0 + 0/3)*sqrt(3)});
    \draw (0,{(4 + 2/3)*sqrt(3)}) -- (12,{(1 - 1/3)*sqrt(3)});
    \draw (0,{(5 + 1/3)*sqrt(3)}) -- (12,{(1 + 1/3)*sqrt(3)});
    \draw (0,{(6 + 0/3)*sqrt(3)}) -- (12,{(2 + 0/3)*sqrt(3)});
    \draw (2,{(6 + 0/3)*sqrt(3)}) -- (12,{(3 - 1/3)*sqrt(3)});
    \draw (4,{(6 + 0/3)*sqrt(3)}) -- (12,{(3 + 1/3)*sqrt(3)});
    \draw (6,{(6 + 0/3)*sqrt(3)}) -- (12,{(4 + 0/3)*sqrt(3)});
    \draw (8,{(6 + 0/3)*sqrt(3)}) -- (12,{(5 - 1/3)*sqrt(3)});
    \draw (10,{(6 + 0/3)*sqrt(3)}) -- (12,{(5 + 1/3)*sqrt(3)});
    \draw (0,{6*sqrt(3)}) -- (0,{-sqrt(3)});
    \draw (1,{6*sqrt(3)}) -- (1,{-sqrt(3)});
    \draw (2,{6*sqrt(3)}) -- (2,{-sqrt(3)});
    \draw (3,{6*sqrt(3)}) -- (3,{-sqrt(3)});
    \draw (4,{6*sqrt(3)}) -- (4,{-sqrt(3)});
    \draw (5,{6*sqrt(3)}) -- (5,{-sqrt(3)});
    \draw (6,{6*sqrt(3)}) -- (6,{-sqrt(3)});
    \draw (7,{6*sqrt(3)}) -- (7,{-sqrt(3)});
    \draw (8,{6*sqrt(3)}) -- (8,{-sqrt(3)});
    \draw (9,{6*sqrt(3)}) -- (9,{-sqrt(3)});
    \draw (10,{6*sqrt(3)}) -- (10,{-sqrt(3)});
    \draw (11,{6*sqrt(3)}) -- (11,{-sqrt(3)});
    \draw (12,{6*sqrt(3)}) -- (12,{-sqrt(3)});
    
    \draw [line width=0.5mm] (1,{(1 + 2/3)*sqrt(3)}) -- (2,{(1 + 1/3)*sqrt(3)}) -- (3, {(1 + 2/3)*sqrt(3)}) -- (3, {(2 + 1/3)*sqrt(3)}) -- (2, {(2 + 2/3)*sqrt(3)}) -- (1, {(2 + 1/3)*sqrt(3)}) -- (1,{(1 + 2/3)*sqrt(3)});
    \end{tikzpicture}
\end{figure}

\end{document}